\documentclass{amsart}

\usepackage{amssymb}
\usepackage{amsmath,amsthm,amsfonts,amssymb,latexsym,mathrsfs,color,hyperref}
\usepackage{graphicx}
\usepackage{xspace}
\usepackage{multirow}
\usepackage{diagbox}
\usepackage{capt-of}
\usepackage{tikz}
\tikzset{
	level 1/.style = {sibling distance = 1.5cm},
	level 2/.style = {sibling distance = 0.8cm},
    level distance = 0.9 cm
}
\usetikzlibrary {decorations.pathmorphing}
\usetikzlibrary{matrix, positioning}
\tikzstyle{snakeline} = [decorate, decoration={snake, amplitude=.4mm, segment length=2mm}]

\usepackage{tikz-cd}
\usepackage[boxsize=1.3em, centerboxes]{ytableau}

\usepackage{tikz-qtree}
\tikzset{every tree node/.style={minimum width=0.1cm,draw,circle},
         blank/.style={draw=none},
         edge from parent/.style=
         {draw,edge from parent path={(\tikzparentnode) -- (\tikzchildnode)}},
         level distance=0.8cm}

\newtheorem{theorem}{Theorem}
\newtheorem{corollary}[theorem]{Corollary}
\newtheorem{proposition}[theorem]{Proposition}

\newtheorem{lemma}[theorem]{Lemma}
\newtheorem{definition}[theorem]{Definition}
\newtheorem{example}[theorem]{Example}
\newtheorem{problem}[theorem]{Problem}

\newtheorem*{LSp}{Label Schema}

\newcommand{\dd}{{\rm dd}}
\newcommand{\val}{{\rm val}}
\newcommand{\ipk}{{\rm ipk}}

\newcommand{\YWCT}{{\rm YWCT\,}}
\newcommand{\SYT}{{\rm SYT\,}}

\newcommand{\LS}{{\rm LS\,}}

\newcommand{\fap}{{\rm fap\,}}

\newcommand{\plat}{{\rm plat\,}}
\newcommand{\ap}{{\rm ap\,}}
\newcommand{\lap}{{\rm lap\,}}
\newcommand{\lpk}{{\rm lpk\,}}

\newcommand{\des}{{\rm des\,}}

\newcommand{\exc}{{\rm exc\,}}

\newcommand{\cyc}{{\rm cyc\,}}

\newcommand{\man}{\mathfrak{A}_n}

\newcommand{\msn}{\mathfrak{S}_n}
\newcommand{\men}{\mathcal{E}_n}

\newcommand{\lrf}[1]{\lfloor #1\rfloor}

\newcommand{\mbn}{{\mathcal B}_n}

\newcommand{\mqn}{\mathcal{Q}_n}

\newcommand{\asc}{{\rm asc\,}}

\newcommand{\Eulerian}[2]{\genfrac{<}{>}{0pt}{}{#1}{#2}}
\newcommand{\Stirling}[2]{\genfrac{\{}{\}}{0pt}{}{#1}{#2}}
\newcommand{\stirling}[2]{\genfrac{[}{]}{0pt}{}{#1}{#2}}
\newcommand{\arxiv}[1]{\href{http://arxiv.org/abs/#1}{\texttt{arXiv:#1}}}
\numberwithin{equation}{section}

\begin{document}

\title{Differential operators, grammars and Young tableaux}

\author{Shi-Mei Ma}
\address{School of Mathematics and Statistics,
        Northeastern University at Qinhuangdao,
         Hebei 066000, P.R. China}
\email{shimeimapapers@163.com}
\thanks{The first author was supported by the National Natural Science Foundation of China (Grant number 12071063).}

\author{Jean Yeh}
\address{Department of Mathematics, National Kaohsiung Normal University, Kaohsiung 82444, Taiwan}
\email{chunchenyeh@nknu.edu.tw}
\thanks{The second author was supported by the National Science Council of Taiwan (Grant number: MOST 110-2115-M-017-002-MY2).}

\author{Yeong-Nan Yeh}
\address{College of Mathematics and Physics, Wenzhou University, Wenzhou 325035, P.R. China}
\email{mayeh@math.sinica.edu.tw}

\subjclass[2020]{Primary 05A19; Secondary 05E10}



\keywords{Differential operators, Standard Young tableaux, Increasing trees, Normal ordered problems,
Eulerian polynomials, Narayana polynomials, Ramanujan polynomials, $e$-positivity}

\begin{abstract}
In algebraic combinatorics and formal calculation, context-free grammar is defined by a formal derivative based on a set of substitution rules.
In this paper, we investigate this issue from three related viewpoints.
Firstly, we introduce a differential operator method. As one of the applications, we deduce a new grammar for the Narayana polynomials.
Secondly, we investigate the normal ordered
grammars associated with Eulerian polynomials.
Thirdly, motivated by the theory of differential posets,
we introduce a box sorting algorithm which leads to a bijection between
the terms in the expansion of $(cD)^nc$ and a kind of ordered weak set partitions,
where $c$ is a smooth function in the indeterminate $x$ and $D$ is the derivative with respect to $x$.
Using a map from ordered weak set partitions to standard Young tableaux, we find an expansion of
$(cD)^nc$ in terms of standard Young tableaux. Combining this with the theory of context-free grammars, we provide a unified
interpretations for the Ramanujan polynomials, Andr\'e polynomials, left peak polynomials,
interior peak polynomials, Eulerian polynomials of types $A$ and $B$, $1/2$-Eulerian polynomials, second-order Eulerian polynomials, and Narayana polynomials of types $A$ and $B$
in terms of standard Young tableaux. Along the same lines, we present an expansion of the powers of $c^kD$ in terms of
standard Young tableaux,
where $k$ is a positive integer.
In particular, we provide four interpretations for the second-order Eulerian polynomials.
All of the above apply to the theory of formal differential operator rings.
\end{abstract}

\maketitle
\tableofcontents
\section{Introduction}
In recent years, one can witness the spread of the idea of a context-free grammar (see~\cite{Chen23,ChenYang21,Dumont96,Lin21,Ma1902,Ma23}), which is a symbolic method based on a set of
production rules and it is closely related to umbral calculus and computer algebra. Context-free grammar has emerged as a
basic tool in combinatorics and formal calculation, leading to several new constructing bijections.
We shall investigate this issue from three closely related viewpoints:
\begin{itemize}
  \item Introduce a differential operator method;
  \item Introduce the theory of normal ordered grammars;
  \item Develop a general method to express Eulerian-type polynomials in terms of standard Young tableaux.
\end{itemize}

The {\it Weyl algebra} $W$ is the unital algebra generated by two symbols $D$ and $U$ satisfying the commutation relation
$DU-UD=I$,
where $I$ is the identity which we identify with ``1". In other words, $W=\langle D,U |DU-UD=I \rangle$.
Any word $w$ in the letters $U,V$ can always be brought
into {\it normal ordered form} where all letters $D$ stand to the right of all the letters $U$.
A famous example of $W$ is given by the substitution:
$D\rightarrow \frac{\mathrm{d}}{\mathrm{d}x},~U\rightarrow x$. Except where otherwise indicated, we always let $D=\frac{\mathrm{d}}{\mathrm{d}x}$.
The expansion of $(xD)^n$ has been studied as early as 1823 by Scherk~\cite[Appendix~A]{Blasiak10}. He found that
\begin{equation}\label{Stirling-def}
(xD)^n=\sum_{k=0}^n\Stirling{n}{k}x^kD^k,
\end{equation}
where $\Stirling{n}{k}$ is the {\it Stirling number of the second kind}, i.e., the number of partitions of the set $[n]=\{1,2,\ldots,n\}$ into $k$ blocks (non-empty subsets).
According to~\cite[Proposition~A.2]{Blasiak10}, one has
\begin{equation}\label{stirling-def}
(\mathrm{e}^xD)^n=\mathrm{e}^{nx}\sum_{k=0}^n\stirling{n}{k}D^k,
\end{equation}
where $\stirling{n}{k}$ is the (signless) Stirling number of the first
kind, i.e.,
the number of permutations of $[n]$ with $k$ cycles.
Many generalizations and variations of~\eqref{Stirling-def} and~\eqref{stirling-def} occur naturally in quantum physics, combinatorics and algebra.
The reader is referred to Blasiak-Flajolet~\cite{Blasiak10} and Schork~\cite{Schork21} for surveys and~\cite{Agapito21,Briand20,Mansour10,Eu17,Han23,Mansour16} for recent progress
on this subject.

In~\cite{Stanley88}, Stanley introduced $r$-differential posets, as a class of ranked posets enjoying
a number of important combinatorial and algebraic properties. A basic tool in the theory of differential
posets is the use of two adjoint operators $D$ and $U$ on the vector space of linear combinations of elements of $P$.
Let $P$ be a locally finite graded poset with minimal element.
According to~\cite[Theorem~2.2]{Stanley88},
$P$ is differential if and only if $DU-UD=I$.
Many of the properties of Young's lattice can be deduced
from this relation. Many authors have since provided generalizations
of differential poset, see~\cite{Lopes23} for instance.

Let $Y$ denote Young's lattice of partitions, i.e. the set of all
integer partitions ordered by inclusion, and denote by $\mathbb{F}Y$ the vector space of formal linear
combinations of elements of $Y$. As operators on $\mathbb{F}Y$, it is well known that $U$ acts by sending
a partition $\lambda$ to the sum of all the partitions which can be obtained
from $\lambda$ by adding $1$ to a part of $\lambda$ (without changing the property that the parts are weakly
decreasing) or creating a new part $1$ at the end, and $D$ acts similarly but by subtracting $1$ (see~\cite{Benkart98,Lopes23,Stanley90} for details).
Motivated by this idea, we shall introduce a box sorting algorithm in Section~\ref{sec05}.

Following Chen~\cite{Chen93}, a {\it context-free grammar} $G$ over an alphabet
$V$ is defined as a set of substitution rules replacing a letter in $V$ by a formal function over $V$.
The formal function may be a polynomial or a Laurent polynomial.
The formal derivative $D_G$ with respect to $G$ satisfies the derivation rules:
$$D_G(u+v)=D_G(u)+D_G(v),~D_G(uv)=D_G(u)v+uD_G(v).$$
So the {\it Leibniz rule} holds:
$$D_G^n(uv)=\sum_{k=0}^n\binom{n}{k}D_G^k(u)D_G^{n-k}(v).$$

In this paper, we always let $D_G$ be the formal derivative associated with the grammar $G$. Here we
recall two results, which may be seen as the correspondences of~\eqref{Stirling-def} and~\eqref{stirling-def}.
\begin{proposition}[{\cite{Chen93}}]\label{grammar01}
If $G=\{a\rightarrow ab, b\rightarrow b\}$, then $D_{G}^n(a)=a\sum_{k=0}^n\Stirling{n}{k}b^k$.
\end{proposition}
\begin{proposition}[{\cite{Ma131}}]\label{grammar02}
If $G=\{a\rightarrow ab, b\rightarrow bc, c\rightarrow c^2\}$, then
$D_{G}^n(a)=a\sum_{k=0}^n\stirling{n}{k}b^kc^{n-k}$.
\end{proposition}

In~\cite{Chen93}, Chen applied grammar to obtain demonstrations of Fa\`{a} di Bruno's formula, and some identities
concerning Stirling numbers and symmetric functions.
In~\cite{Dumont96}, Dumont obtained grammatical interpretation for
Eulerian polynomials, Roselle polynomials, Andr\'e polynomials and cycle indicator polynomials.
A remarkable discovery of Dumont-Ramamonjisoa~\cite{Dumont9602} is the grammar for the Ramanujan polynomials, in connection
with Shor's refinement of Cayley's formula.
We now describe two widely used methods.
Following~\cite{Chen17,Dumont96}, the {\it grammatical labeling method} is an assignment of the underlying elements of
a combinatorial structure with variables, which is consistent with the substitution rules
of a grammar.
Another method is the {\it change of grammars}, which essentially is a change of variables, see~\cite{Chen2301,Chen23,Lin21,Ma1902,Ma23} for applications.

The following result suggests that it is natural to consider normal ordered form of grammars.
\begin{proposition}\label{prop-3}
If $G=\{x\rightarrow 1\}$, then one has $\left(xD_G\right)^n=\sum_{k=0}^n\Stirling{n}{k}x^kD_{G}^k$.
\end{proposition}

Following Hwang-Chern-Duh~\cite{Hwang20}, we say that $\{\mathcal{P}_{n}(x)\}_{n\geqslant 0}$ is a sequence of {\it Eulerian-type polynomials} if
it satisfies the general Eulerian recurrence:
\begin{equation*}\label{Eulerian02}
\mathcal{P}_{n}(x)=(\alpha(x)n+\gamma(x))\mathcal{P}_{n-1}(x)+\beta(x)(1-x)\frac{\mathrm{d}}{\mathrm{d}x}\mathcal{P}_{n-1}(x)
\end{equation*}
for $n\geqslant 1$, where $\mathcal{P}_0(x),\alpha(x),\beta(x)$ and $\gamma(x)$ are given functions of $x$.
When $\alpha(x)=x,~\gamma(x)=0$ and $\beta(x)=x$, the polynomial $\mathcal{P}_{n}(x)$ reduces to the classical Eulerian polynomial $A_n(x)$.
The well studied Eulerian-type polynomials including Eulerian polynomials of type $B$, Narayana polynomials and second-order Eulerian polynomials.
As given by~\eqref{AnxSYT01}, Eulerian polynomial has a famous interpretation in terms of standard Young tableaux, which leads to a general problem.
\begin{problem}\label{problem}
How to establish the relationships between Eulerian-type polynomials and Young tableaux?
In other words, could we provide a general method to express Eulerian-type polynomials in terms of Young tableaux?
\end{problem}
The organization of this paper is as follows.
In Section~\ref{section02}, we collect the definitions and preliminaries results that will be used in the rest part of this paper.
In Section~\ref{sec03}, we introduce a differential operator method and we propose three new applications of this method.
In Section~\ref{section04},
we study normal ordered grammars associated with Eulerian polynomials.
In Section~\ref{sec05}, we introduce the box sorting algorithm, and then we present Theorem~\ref{TSYTncol},
which gives an answer to Problem~\ref{problem}.
In Section~\ref{sec06}, we collect several applications of Theorem~\ref{TSYTncol}.
In particular, in Theorems~\ref{thm-Cnxy},~\ref{cnx01},~\ref{cnx02},~\ref{cnx03}, we present
four interpretations of the second-order Eulerian polynomials.
Moreover, in Theorem~\ref{NBthm}, a deep connection between the type $B$ Narayana polynomials and the type $B$ Eulerian polynomials is established.
\section{Notation and Preliminaries}\label{section02}
\subsection{A brief survey of the powers of differential operators}\label{subsection21}
\hspace*{\parindent}

Given the exponential series
$$f(t)=x_0+x_1t+x_2\frac{t^2}{2!}+\cdots+x_n\frac{t^n}{n!}+\cdots.$$
In 1857, Cayley~\cite{Cayley57} introduced trees in order to solve the differential equation:
\begin{equation}\label{Vxt}
\frac{\mathrm{d}y}{\mathrm{d}t}(t)=f(y(t)),~y(0)=0.
\end{equation}
Under various guises, the differential equation~\eqref{Vxt} frequently appeared in combinatorics, probability, special functions, algebra and representation theory.
The reader is referred to Bergeron-Flajolet-Salvy~\cite{Bergeron92}, Hivert-Novelli-Thibon~\cite{Hivert06} and Lopes~\cite{Lopes23} for surveys on this subject.

From~\eqref{Vxt}, we see that
$$\frac{\mathrm{d}y}{\mathrm{d}t}(t)=x_0+x_1y(t)+x_2\frac{{y(t)}^2}{2!}+\cdots+x_n\frac{{y(t)}^n}{n!}+\cdots,~y(0)=0.$$
Dumont~\cite[Section~3.3]{Dumont96} observed that its solution is given by the series
\begin{equation}\label{yx0}
y=\sum_{n=1}^\infty D_{G}^{n-1}(x_0)\frac{t^n}{n!},~\text{where $G=\{x_i\rightarrow x_0x_{i+1}~{\text{for $i\geqslant 0$}}\}$.}
\end{equation}

Throughout this paper, we always let $c:=c(x)$ and $f:=f(x)$ be two smooth functions in the indeterminate $x$.
We adopt the convention that $c_k=D^kc$ and $\mathbf{f}_k=D^kf$ for $k\geqslant 0$. Set $\mathbf{f}_0=f$ and $c_0=c$, where $D=\frac{\mathrm{d}}{\mathrm{d}x}$.
The first few $(cD)^nf$ are given as follows:
\begin{align*}
(cD)f&=(c)  {\mathbf{f}}_1,\\
(cD)^2f&=(c c_1 )  {\mathbf{f}}_1 +(c^2 )  {\mathbf{f}}_2,\\
(cD)^3f&=(c c_1^2  +c^2 c_2 )  {\mathbf{f}}_1 +(3c^2 c_1 )  {\mathbf{f}}_2 +(c^3 )  {\mathbf{f}}_3,\\
(cD)^4f&=(c c_1^3  +4c^2 c_1 c_2  +c^3 c_3 )  {\mathbf{f}}_1 +(7c^2 c_1^2  +4c^3 c_2 )  {\mathbf{f}}_2 +(6c^3 c_1 )  {\mathbf{f}}_3 +(c^4 )  {\mathbf{f}}_4.
\end{align*}
\begingroup\vspace*{-\baselineskip}
\captionof{table}{Expansions of $(cD)^nf$}\label{tab:1}
\vspace*{\baselineskip}\endgroup
For $n\geqslant 1$, we define
\begin{equation}\label{Ank-def}
(cD)^nf =\sum_{k=1}^nF_{n,k}\mathbf{f}_k.
\end{equation}
From Table~\ref{tab:1}, we note that $F_{n,k}=F_{n,k}(c,c_1,\ldots,c_{n-k})$ is a function of $c,c_1,\ldots,c_{n-k}$.
In particular, $F_{1,1}=c$, $F_{2,1}=cc_1$ and $F_{2,2}=c^2$.
By induction, it is easy to verify that
$F_{n+1,1}=cDF_{n,1}$, $F_{n,n}=c^n$ and for $2\leqslant k\leqslant n$, we have
$F_{n+1,k}=cF_{n,k-1}+cDF_{n,k}$.
Let $G$ be given in~\eqref{yx0}. It is easy to verify that
$$F_{n,1}(x_0,x_1,\ldots,x_{n-1})=\left(cD_G\right)^{n-1}c\big{|}_{c=x_0,c_1=x_1,\ldots,c_{n-1}=x_{n-1}}=D_{G}^{n-1}(x_0)~{\text{for $n\geqslant 1$}}.$$

By induction, Comtet~\cite{Comtet73} found an explicit formula of $F_{n,k}$.
Very recently, Han-Ma~\cite{Han23} first gave a simple proof of Comtet's formula via inversion sequences,
and then introduced $k$-Young tableaux and their $g$-indices. Using the indispensable $k$-Young tableaux, they obtained
combinatorial interpretations of Eulerian polynomials and second-order Eulerian
polynomials in terms of standard Young tableaux.
In Section~\ref{sec05}, we shall give a substantial and original improvement of this idea.

A {\it partition} of $n$ is a weakly decreasing sequence of nonnegative integers:
$$\lambda=(\lambda_1,\lambda_2,\ldots,\lambda_{\ell}),$$
where $\sum_{i=1}^{\ell}\lambda_i=n$. Each $\lambda_i$ is called a {\it part} of $\lambda$.
If $\lambda$ is a partition of $n$, then we write $\lambda\vdash n$.
As usual, we denote by $m_i$ the number of parts equals $i$. By using the multiplicities, we also denote $\lambda$ by $(1^{m_1}2^{m_2}\cdots n^{m_n})$.
The partition with all parts equal to $0$ is the empty partition. The {\it length} of $\lambda$, denoted $\ell(\lambda)$, is the maximum subscript $j$ such that $\lambda_j>0$.
The {\it Ferrers diagram} of $\lambda$ is graphical representation of $\lambda$ with $\lambda_i$ boxes in its $i$th row and the boxes are left-justified.
For a Ferrers diagram $\lambda\vdash n$ (we will often identify a partition with its Ferrers
diagram), a {\it standard Young tableau} ($\SYT$, for short) of shape $\lambda$ is a filling of the $n$ boxes of $\lambda$ with the integers $\{1,2,\ldots, n\}$ such that each of those
integers is used exactly once, and all rows and columns are increasing (from left to right, and from bottom to top, respectively).
Given a Young tableau, we number its rows starting from the bottom and going above.

Let $\operatorname{SYT}_\lambda$ be the set of standard Young tableaux of shape $\lambda$.
Set $$\SYT(n)=\bigcup_{\lambda\vdash n}\operatorname{SYT}_\lambda.$$
The {\it descent set} for $T\in\SYT(n)$ is defined by
$$\operatorname{Des}(T)=\{i\in[n-1]: i+1~{\text{appears in an upper row of $T$ than $i$}} \}.$$
Denote by $\#V$ the cardinality of a set $V$.
Let $\des(T)=\#\operatorname{Des}(T)$.
The Robinson-Schensted correspondence is a bijection from permutations to pairs of standard Young tableaux of the
same shape. This correspondence and its generalization, the
Robinson-Schensted-Knuth correspondence, have become centrepieces of enumerative and algebraic combinatorics due to their many applications and properties, see~\cite{Sagan01} for details.
A remarkable application of the Robinson-Schensted correspondence is the following identity:
\begin{equation}\label{AnxSYT01}
A_n(x)=\sum_{\pi\in\msn}x^{\des(\pi)+1}=\sum_{\lambda\vdash n}\#\operatorname{SYT}_\lambda\sum_{T\in\operatorname{SYT}_\lambda}x^{\des(T)+1},
\end{equation}
which has been extended to skew shapes, see~\cite{Adin19,Huang20} for instances. In Section~\ref{sec06}, we shall express
ten kinds of well studied polynomials in terms of standard Young tableaux.

Recently, Briand-Lopes-Rosas~\cite{Briand20} studied the normally ordered form of the operator $cD^d$.
In~\cite[Corollary~3.11]{Briand20}, they presented a generalization of Comtet's explicit formula for $F_{n,k}$.
Moreover, they gave a survey on the combinatorial and arithmetic properties of the coefficients $F_{n,k}$, which can be summarized as follows.
\begin{proposition}[{\cite{Briand20}}]\label{prop02}
Let $F_{n,k}$ be defined by~\eqref{Ank-def}.
There exist positive integers $a(n,\lambda)$ such that
\begin{equation*}
F_{n,k}=\sum_{\lambda\vdash n-k} a(n,\lambda)c^{n-\ell(\lambda)}c_{\lambda},
\end{equation*}
where $\lambda$ runs over all partitions of $n-k$. The Stirling numbers of the first and second kinds, and the Eulerian numbers can be respectively expressed as follows:
\begin{equation}\label{Fnk}
\stirling{n}{k}=\sum_{\lambda\vdash n-k} a(n,\lambda),~
\Stirling{n}{k}=a(n,1^{n-k}),~
\Eulerian{n}{k}=\sum_{\ell(\lambda)=n-k} a(n,\lambda).
\end{equation}
\end{proposition}
It should be noted that a dual of~\eqref{Fnk} is given by Corollary~\ref{Dual-result}.
\subsection{Eulerian-type polynomials and differential expressions}\label{subsection22}
\hspace*{\parindent}

Many polynomials can be generated by successive differentiations of a given base function, see~\cite[Section~4.3]{Hwang20} for a survey.
In this subsection, we collect five differential expressions, which will be used in Section~\ref{sec03}.

The {\it (type $A$) Eulerian polynomials} $A_n(x)$ can be defined by the differential expression:
\begin{equation}\label{Anx-poly-def}
\left(x\frac{\mathrm{d}}{\mathrm{d}x}\right)^n\frac{1}{1-x}=\sum_{k=0}^\infty k^nx^k=\frac{A_n(x)}{(1-x)^{n+1}}.
\end{equation}
They satisfy the recurrence relation
\begin{equation}\label{Eulerian01}
A_{n}(x)=nxA_{n-1}(x)+x(1-x)\frac{\mathrm{d}}{\mathrm{d}x}A_{n-1}(x),~A_0(x)=1.
\end{equation}
Let $\msn$ be the {\it symmetric group} of all permutations of $[n]$. For $\pi=\pi(1)\pi(2)\cdots\pi(n)\in\msn$,
the index $i\in [n-1]$ is a {\it descent} if $\pi(i)>\pi(i+1)$.
The {\it Eulerian polynomials} can also be defined by $$A_n(x)=\sum_{\pi\in\msn}x^{\des(\pi)+1}=\sum_{k=1}^n\Eulerian{n}{k}x^k,$$
where $\Eulerian{n}{k}$ are known as the {\it Eulerian numbers} (see~\cite{Branden08,Bre94,Gessel20,Petersen15}).
It follows from~\eqref{Eulerian01} that
\begin{equation}\label{Eulerian02}
\Eulerian{n}{k}=k\Eulerian{n-1}{k}+(n-k+1)\Eulerian{n-1}{k-1}.
\end{equation}

Savage-Viswanathan~\cite{Savage12} introduced the {\it $1/k$-Eulerian polynomials} $A_n^{(k)}(x)$, which can be defined by any of the following relations:
\begin{equation*}\label{Ankx-def01}
\sum_{n=0}^\infty A_n^{(k)}(x)\frac{z^n}{n!}=\left(\frac{1-x}{\mathrm{e}^{kz(x-1)}-x} \right)^{\frac{1}{k}},
\end{equation*}
\begin{equation}\label{Ankx-def04}
\sum_{t\geqslant 0}\binom{t-1+\frac{1}{k}}{t}(kt)^nx^t=\frac{ x^nA_n^{(k)}(1/x)}{(1-x)^{n+\frac{1}{k}}}.
\end{equation}
In particular, $A_n(x)=x^nA_n^{(1)}(1/x)$.
By~\eqref{Ankx-def04}, we find that
\begin{equation}\label{Ankx-def05}
\left(kx\frac{\mathrm{d}}{\mathrm{d}x}\right)^n\frac{1}{(1-x)^{1/k}}=\frac{x^nA_n^{(k)}(1/x)}{(1-x)^{n+\frac{1}{k}}}.
\end{equation}
It is now well known (see~\cite{Savage15}) that
\begin{equation}\label{Ankx-def02}
A_{n}^{(k)}(x)=\sum_{\pi\in\msn}x^{\exc(\pi)}k^{n-\cyc(\pi)},
\end{equation}
where $\exc(\pi)=\#\{i\in[n-1]: \pi(i)>i\}$ and $\cyc(\pi)$ is the number of cycles of $\pi$.
The $1/k$-Eulerian polynomial have been extensively studied in recent years, see~\cite{Hwang20,Ma15,Ma23}.

Let $\pm[n]=[n]\cup\{\overline{1},\overline{2},\ldots,\overline{n}\}$, where $\overline{i}=-i$.
Let $\mbn$ be the hyperoctahedral group of rank $n$. Elements of $\mbn$ are signed permutations $\sigma$ of $\pm[n]$ such that $\sigma(-i)=-\sigma(i)$ for all $i$.
One can ignore the negative index of $\sigma$ and just write $\sigma=\sigma(1)\sigma(2)\cdots\sigma(n)$.
The {\it type $B$ Eulerian polynomials} are defined by
\begin{equation*}
B_n(x)=\sum_{\sigma\in\mbn}x^{\operatorname{des}_B(\sigma)},
\end{equation*}
where $\operatorname{des}_B(\sigma):=\#\{i\in [0,n-1]:~\sigma(i)>\sigma({i+1}),~\sigma(0)=0\}$.
Following~\cite[p.~29]{Hwang20}, one has
\begin{equation}\label{dxdy}
\left(\frac{\mathrm{d}}{\mathrm{d}y}\right)^n\frac{\mathrm{e}^y}{1-\mathrm{e}^{2y}}=\frac{\mathrm{e}^xB_n(\mathrm{e}^{2x})}{(1-\mathrm{e}^{2x})^{n+1}}.
\end{equation}
Note that $\frac{\mathrm{d}}{\mathrm{d}y}=\frac{\mathrm{dx}}{\mathrm{d}y}\frac{\mathrm{d}}{\mathrm{d}x}$ and $x=\mathrm{e}^y$ is the solution of $x=\frac{\mathrm{d}x}{\mathrm{d}y}$.
It follows from~\eqref{dxdy} that
\begin{equation}\label{Bnxsum}
\left(x\frac{\mathrm{d}}{\mathrm{d}x}\right)^n\frac{x}{1-x^2}=\frac{xB_n(x^2)}{(1-x^2)^{n+1}}.
\end{equation}

The development of the theories of the second-order Eulerian polynomials began with the work of
Buckholtz~\cite{Buckholtz} in his studies of an asymptotic expansion.
For each positive integer $n$ and each complex number $x$,
one can define $S_n(x)$ by the equation
$\mathrm{e}^{nx}=\sum_{r=0}^n\frac{(nx)^r}{r!}+\frac{(nx)^n}{n!}S_n(x)$.
Buckholtz~\cite{Buckholtz} found that
$S_n(x)=\sum_{r=0}^{k-1}\frac{1}{n^r}U_r(x)+O(n^{-k})$,
and
\begin{equation}\label{eqr}
U_r(x)=(-1)^r\left(\frac{x}{1-x}\frac{\mathrm{d}}{\mathrm{d}x}\right)^r\frac{x}{1-x}=(-1)^r\frac{C_r(x)}{(1-x)^{2r+1}},
\end{equation}
where $C_r(x)$ is now called {\it second-order Eulerian polynomial}.

Following Gessel-Stanley~\cite{Gessel78},
a {\it Stirling permutation} of order $n$ is a permutation of the multiset $\{1^2,2^2,\ldots,n^2\}$ such that for each $i$, $1\leqslant i\leqslant n$,
all entries between any two occurrences of $i$ are at least $i$. The reader is referred to~\cite{Chen22,Ma1902,Ma23} for recent
progress on this subject.
Let $\mqn$ be the set of Stirling permutations of order $n$, and let $\sigma=\sigma_1\sigma_2\cdots\sigma_{2n}\in\mqn$.
The numbers of {\it descents}, {\it ascent-plateaux}, {\it left ascent-plateaux} and {\it flag ascent-plateaux} of $\sigma$ are defined by
\begin{align*}
\des(\sigma)&=\#\{i\in[2n]:~\sigma_i>\sigma_{i+1},~\sigma_{2n+1}=0\},\\
\ap(\sigma)&=\#\{i\in[2,2n-1]:~\sigma_{i-1}<\sigma_i=\sigma_{i+1}\},\\
\lap(\sigma)&=\#\{i\in[2n-1]:~\sigma_{i-1}<\sigma_i=\sigma_{i+1},~\sigma_0=0\},\\
\fap(\sigma)&=\ap(\sigma)+\lap(\sigma).
\end{align*}
Define
\begin{align*}
C_n(x)=\sum_{\sigma\in\mqn}x^{\des(\pi)},~
T_n(x)=\sum_{\sigma\in\mqn}x^{\fap(\sigma)}.
\end{align*}
Below are these polynomials for $n\leqslant 3$:
\begin{align*}
C_1(x)&=x,~
C_2(x)=x+2x^2,~
C_3(x)=x+8x^2+6x^3;\\
T_1(x)&=x,~
T_2(x)=x+x^2+x^3,~
T_3(x)=x+3x^2+7x^3+3x^4+x^5.
\end{align*}

According to~\cite[p.~14]{Ma20}, the {\it flag ascent-plateau polynomials} $T_n(x)$ can be defined as follows:
\begin{equation}\label{Rnx-altrun}
\left(x\frac{\mathrm{d}}{\mathrm{d}x}\right)^{n}r(x)=\frac{r(x)T_{n}(x)}{(1-x^2)^{n}},~\text{where}~r(x)=\sqrt{\frac{1+x}{1-x}}.
\end{equation}

In conclusion, we collect the differential expressions for five kinds of polynomials.
These polynomials share similar properties,
including real-rootedness~\cite{Ma22}, combinatorial expansions~\cite{Lin21,Ma1902,Ma23} and asymptotic distributions~\cite{Hwang20}.
Except the type $A$ Eulerian polynomials,
here we provide the recursions for the other four kinds of polynomials:
\begin{equation*}
\begin{split}
A_{n+1}^{(k)}(x)&=(1+nkx)A_{n}^{(k)}(x)+kx(1-x)\frac{\mathrm{d}}{\mathrm{d}x}A_{n}^{(k)}(x),~A_{0}^{(k)}(x)=1;\\
B_{n}(x)&=(1+(2n-1)x)B_{n-1}(x)+2x(1-x)\frac{\mathrm{d}}{\mathrm{d}x}B_{n-1}(x),~B_0(x)=1;\\
C_{n}(x)&=(2n-1)xC_n(x)+x(1-x)\frac{\mathrm{d}}{\mathrm{d}x}C_n(x),~C_0(x)=1;\\
T_{n+1}(x)&=(x+2nx^2)T_n(x)+x(1-x^2)\frac{\mathrm{d}}{\mathrm{d}x}T_n(x),~T_0(x)=1.
\end{split}
\end{equation*}
\section{A differential operator method}\label{sec03}
In order to illustrate the basic idea of a differential operator method, we first give seven examples.
We then deduce new grammars for three kinds of polynomials.
\subsection{The basic idea of differential operator method}\label{sec0301}
\hspace*{\parindent}

A connection between the formula~\eqref{Stirling-def} and Proposition~\ref{grammar01} is given as follows.
\begin{example}
It is clear from~\eqref{Stirling-def} that $\left(x\frac{\mathrm{d}}{\mathrm{d}x}\right)^n\mathrm{e}^x=\mathrm{e}^x\sum_{k=0}^n\Stirling{n}{k}x^k$.
Setting
\begin{equation*}\label{Tab}
T=x\frac{\mathrm{d}}{\mathrm{d}x},~a=\mathrm{e}^x~\text{and}~b=x,
\end{equation*}
we obtain $T(a)=ab$ and $T(b)=b$. Thus $T^n(a)=D_{G}^n(a)$, where $G=\{a\rightarrow ab, b\rightarrow b\}$.
\end{example}

In~\cite{Dumont96}, Dumont obtained the context-free grammar for Eulerian polynomials by using a grammatical labeling of circular permutations.
\begin{proposition}[{\cite[Section~2.1]{Dumont96}}]\label{grammar03}
Let $G=\{a\rightarrow ab, b\rightarrow ab\}$.
Then for $n\geqslant 1$, one has
\begin{equation*}
D_{G}^n(a)=D_{G}^n(b)=b^{n+1}A_n\left(\frac{a}{b}\right).
\end{equation*}
\end{proposition}

A direct connection
between~\eqref{Anx-poly-def} and Proposition~\ref{grammar03} is given by the following example.
\begin{example}\label{ex-Eul}
Setting
\begin{equation*}\label{Tab}
T=x\frac{\mathrm{d}}{\mathrm{d}x},~a=\frac{1}{1-x}~\text{and}~b=\frac{x}{1-x},
\end{equation*}
we obtain $T(a)=T(b)=ab$. So we have $T^n(a)=D_{G}^n(a)$, where $G=\{a\rightarrow ab, b\rightarrow ab\}$. Using~\eqref{Stirling-def}, it follows that
$$T^n(a)=\sum_{k=0}^n\Stirling{n}{k}x^k\left(\frac{\mathrm{d}}{\mathrm{d}x}\right)^k\frac{1}{1-x}=\sum_{k=0}^n\Stirling{n}{k}x^kk!\frac{1}{(1-x)^{k+1}}.$$
Combining the above expansion with~\eqref{Anx-poly-def}, we obtain
$$\frac{A_n(x)}{(1-x)^{n+1}}=\sum_{k=0}^n\Stirling{n}{k}x^kk!\frac{1}{(1-x)^{k+1}}.$$
Multiplying $(1-x)^{n+1}$ on both sides leads to the classical {\it Frobenious formula}:
\begin{equation}\label{Frobenius}
A_n(x)=\sum_{k=0}^nk!\Stirling{n}{k}x^k(1-x)^{n-k}.
\end{equation}
\end{example}

By~\eqref{Ankx-def05}, we now give a natural generalization of Example~\ref{ex-Eul}.
\begin{example}\label{ex2}
Setting $$T=kx\frac{\mathrm{d}}{\mathrm{d}x},~a=\frac{1}{(1-x)^{1/k}}~\text{and}~b=\left(\frac{x}{1-x}\right)^{1/k},$$
we get $T(a)=ab^k,~T(b)=a^kb$. Let $G=\{a\rightarrow ab^k,~b\rightarrow a^kb\}$. We get $D_{G}^n(a)=T^n(a)$.
Note that $ab^{kn}={x^n}/{(1-x)^{n+\frac{1}{k}}}$ and ${a^k}/{b^k}={1}/{x}$.
It follows from~\eqref{Ankx-def05} that
\begin{equation}\label{Tnak-Euler}
D_{G}^n(a)=T^n(a)=\frac{x^nA_n^{(k)}(1/x)}{(1-x)^{n+\frac{1}{k}}}=ab^{kn}A_n^{(k)}\left(\frac{a^k}{b^k}\right).
\end{equation}
\end{example}

\begin{proposition}\label{Ankx}
For $n\geqslant 1$, we have
\begin{equation}\label{Ankx-Explicit}
A_n^{(k)}(x)=\sum_{i=1}^n\Stirling{n}{i}k^{n-i}\prod_{j=0}^{i-1}(1+jk)(x-1)^{n-i}.
\end{equation}
\end{proposition}
\begin{proof}
Combining~\eqref{Stirling-def} and~\eqref{Tnak-Euler}, we find that
$$T^n(a)=\sum_{i=1}^n\Stirling{n}{i}k^nx^i\left(\frac{\mathrm{d}}{\mathrm{d}x}\right)^i\frac{1}{(1-x)^{1/k}}=\sum_{i=1}^n\Stirling{n}{i}k^{n-i}\prod_{j=0}^{i-1}(1+jk)x^i\frac{1}{(1-x)^{\frac{1}{k}+i}}.$$
So we have
\begin{equation}\label{Ankx-exp}
x^nA_n^{(k)}(1/x)=\sum_{i=1}^n\Stirling{n}{i}k^{n-i}\prod_{j=0}^{i-1}(1+jk)x^i(1-x)^{n-i},
\end{equation}
which yields the desired explicit formula.
\end{proof}

We now deduce the grammar for the type $B$ Eulerian polynomials $B_n(x)$.
\begin{example}\label{ex-Bn}
Setting $$T=x\frac{\mathrm{d}}{\mathrm{d}x},~a=\frac{x}{\sqrt{1-x^2}}~\text{and}~b=\frac{1}{\sqrt{1-x^2}},$$
we get $T(a)=ab^2,~T(b)=a^2b$.
Let $G=\{a\rightarrow ab^2,~b\rightarrow a^2b\}$.
It follows from~\eqref{Bnxsum} that
\begin{equation}\label{DG3ab}
D_{G}^n(ab)=T^n(ab)=ab^{2n+1}B_n\left(\frac{a^2}{b^2}\right).
\end{equation}
\end{example}

The grammar for second-order Eulerian polynomials was first discovered by Chen-Fu~\cite{Chen17} by using a
grammatical labeling of Stirling permutations. Here we give a simple derivation.
\begin{example}\label{ExG4}
Setting $$T=\frac{x}{1-x}\frac{\mathrm{d}}{\mathrm{d}x},~a=\frac{x}{1-x}~\text{and}~b=\frac{1}{1-x},$$
we get $T(a)=ab^2$ and $T(b)=ab^2$. Let $G=\{a\rightarrow ab^2,~b\rightarrow ab^2\}$.
By~\eqref{eqr}, we see that $$D_{G}^n(a)=T^n(a)=b^{2n+1}C_n\left(\frac{a}{b}\right).$$
\end{example}

In the following, we deduce the grammar for flag ascent-plateau polynomials $T_n(x)$.
\begin{example}\label{ex11}
Setting $$T=x\frac{\mathrm{d}}{\mathrm{d}x},~a=\frac{x}{\sqrt{1-x^2}},~b=\frac{1}{\sqrt{1-x^2}}~\text{and}~c=a+b=r(x),~|x|<1.$$
we get $T(a)=ab^2,~T(b)=a^2b$ and $T(c)=abc$.
Let $G=\{c\rightarrow abc,~a\rightarrow ab^2,~b\rightarrow a^2b\}$.
It follows from~\eqref{Rnx-altrun} that $$D_{G}^n(c)=T^n(c)=cb^{2n}T_n\left(\frac{a}{b}\right).$$

By the Leibniz rule, we obtain
$$D_{G}^{n+1}(c)=D_{G}^n(abc)=\sum_{k=0}^n\binom{n}{k}D_{G}^k(c)D_{G}^{n-k}(ab).$$
Combining this with~\eqref{DG3ab}, we see that
there is a close connection between the flag ascent-plateau polynomials $T_n(x)$ and the type $B$ Eulerian polynomials $B_n(x)$:
 $$T_{n+1}(x)=x\sum_{k=0}^n\binom{n}{k}T_k(x)B_{n-k}(x^2).$$
\end{example}

The {\it Hermite polynomials} $H_n(x)$ can be defined by any of the following relations:
\begin{equation*}
\begin{split}
H_n(x)&=(-1)^n\mathrm{e}^{x^2}\left(\frac{\mathrm{d}}{\mathrm{d}x}\right)^n\mathrm{e}^{-x^2};\\
H_{n+1}(x)&=2xH_n(x)-\frac{\mathrm{d}}{\mathrm{d}x}H_{n}(x),~H_0(x)=1,~H_1(x)=2x.
\end{split}
\end{equation*}
\begin{example}
Setting $T=-\frac{\mathrm{d}}{\mathrm{d}x},~a=\mathrm{e}^{-x^2}~\text{and}~b=x$,
we get $T(a)=2ab$ and $T(b)=-1$. Let $G=\{a\rightarrow 2ab,~b\rightarrow -1\}$.
By induction, it is routine to verify that
$$D_{G}^n(a)=T^n(a)=aH_n(b).$$
\end{example}
In the past decades, there is a larger literature devoted to variations of Eulerian polynomials, see~\cite{Bre94,Lin21,Petersen15} for instances.
Using the differential operator method, in the following subsections we shall deduce new grammars for three cousins of Eulerian polynomials.
\subsection{Narayana polynomials}\label{sec0302}
\hspace*{\parindent}

Let $1\leqslant r\leqslant s$.
Consider the polynomials $A_{r,s;n}(x)$ defined by
$$\left(x^r\frac{\mathrm{d}^s}{\mathrm{d}x^s}\right)^n\frac{1}{1-x}=\frac{A_{r,s;n}(x)}{(1-x)^{sn+1}}.$$
In~\cite{Agapito14}, Agapito discussed properties of $A_{r,s;n}(x)$, including recursion, symmetry and real-rootedness.
Recently, a combinatorial interpretation of $A_{r,s;n}(x)$ was provided in~\cite[Section~6.1]{Agapito21}.
According to~\cite[Corollary~3.9]{Agapito14}, one has
\begin{equation}\label{NA-def}
\left(x\frac{\mathrm{d}^2}{\mathrm{d}x^2}\right)^n\frac{1}{1-x}=\frac{n!(n+1)!xN(A_{n-1},x)}{(1-x)^{2n+1}}~{\text{for $n\geqslant 1$}},
\end{equation}
where $N(A_{n-1},x)$ is the (type $A$) Narayana polynomial.
Explicitly, $$N(A_{n-1},x)=\sum_{k=0}^{n-1}\frac{1}{n}\binom{n}{k+1}\binom{n}{k}x^k.$$

Setting $$T_1=\frac{\mathrm{d}}{\mathrm{d}x},~T_2=x\frac{\mathrm{d}}{\mathrm{d}x},~a=\frac{1}{1-x},~b=\frac{x}{1-x},$$
we obtain $$T_1(a)=T_1(b)=a^2,~T_2(a)=T_2(b)=ab.$$
Therefore, by~\eqref{NA-def}, we get a new grammar for the Narayana polynomials.
\begin{theorem}\label{thmN}
Let $G_1=\{a\rightarrow a^2,~b\rightarrow a^2\}$ and $G_2=\{a\rightarrow ab,~b\rightarrow ab\}$.
Then one has
\begin{equation*}
\left(D_{G_2}D_{G_1}\right)^na=\left(D_{G_2}D_{G_1}\right)^nb=n!(n+1)!a^{n+1}b^nN\left(A_{n-1},\frac{a}{b}\right)~{\text{for any $n\geqslant 1$}},
\end{equation*}
which implies that $N(A_{n-1},x)$ is symmetric.
\end{theorem}
\subsection{Descent polynomials of multipermutations}\label{sec0303}
\hspace*{\parindent}

For any $n$-element multiset $M$, a {\it multipermutation} of $M$ is a sequence of its elements.
Let
$P(M)$ be the set of multipermutations of $M$.
A {\it descent} of $\sigma\in P(M)$
is an index $i$ such that $\sigma_i>\sigma_{i+1}$, where $i\in[n-1]$.
Let $\des(\sigma)$ be the number of descents of $\sigma$.
Define
\begin{align*}
P(x;M)&=\sum_{\sigma \in P(M)}x^{\des(\sigma)}.
\end{align*}
Clearly, $xP\left(x;[n]\right)=A_n(x)$.
In the past decades, there has been much interest in the polynomials $P(x;M)$. Let $\{1^{m_1},2^{m_2},\ldots,n^{m_n}\}$ be a multiset,
where the element $i$ appears $m_i$ times.
In~\cite{Carlitz78}, Carlitz-Hoggatt studied the recursion of $P\left(x;\{1^p,2^p,\ldots,n^p\}\right)$, where $p$ is a positive integer.
In~\cite[Theorem~3.23]{Savage15}, Savage-Visontai presented a combinatorial interpretation of $P\left(x;\{1^2,2^2,\ldots,n^2\}\right)$ in terms of $s$-inversion sequences.
The reader is referred to~\cite{Agapito21,Lin21,Ma24,Yan2022} for the recent progress on this subject.

A classical result of MacMahon~\cite[Vol~2, Chapter IV, p.~211]{MacMahon20} says that
$$\frac{\sum_{\sigma\in P(\{1^{p_1},2^{p_2},\ldots,n^{p_n}\})}x^{\des(\sigma)}}{(1-x)^{1+\sum_{i=1}^np_i}}
=\sum_{t\geqslant 0}\frac{(t+1)\cdots(t+p_1)\cdots(t+1)\cdots(t+p_n)}{p_1!\cdots p_n!}x^t.$$
When $p_i=p$ for all $i$, this implies
$$\frac{\sum_{\sigma\in P(\{1^{p},2^{p},\ldots,n^{p}\})}x^{\des(\sigma)}}{(1-x)^{1+np}}
=\sum_{t\geqslant 0}{\binom{t+p}{p}}^nx^t.$$
Multiplying $x^p$ on both sides leads to the following identity:
\begin{equation*}
x^p\frac{\sum_{\sigma\in P(\{1^{p},2^{p},\ldots,n^{p}\})}x^{\des\sigma)}}{(1-x)^{1+np}}
=\sum_{t\geqslant p}{\binom{t}{p}}^nx^t,
\end{equation*}
which yields that
\begin{equation}\label{xpd}
\left(x^p\frac{\mathrm{d}^p}{\mathrm{d}x^p}\right)^n\frac{1}{1-x}=x^p(p!)^n\frac{\sum_{\pi\in P(\{1^{p},2^{p},\ldots,n^{p}\})}x^{\des(\pi)}}{(1-x)^{1+np}}.
\end{equation}
When $p=2$, setting $$T_1=\frac{\mathrm{d}}{\mathrm{d}x},~T_2=x^2\frac{\mathrm{d}}{\mathrm{d}x},~a=\frac{1}{1-x},~b=\frac{x}{1-x},$$
we obtain $$T_1(a)=T_1(b)=a^2,~T_2(a)=T_2(b)=b^2.$$
From~\eqref{xpd}, we immediately get the following result.
\begin{theorem}\label{thmP}
Let $G_1=\{a\rightarrow a^2,~b\rightarrow a^2\}$ and $G_2=\{a\rightarrow b^2,~b\rightarrow b^2\}$.
Then we have
\begin{equation*}
\left(D_{G_2}D_{G_1}\right)^na=\left(D_{G_2}D_{G_1}\right)^nb=2^nab^{2n}P\left(\frac{a}{b};\{1^2,2^2,\ldots,n^2\}\right)~{\text{for $n\geqslant 1$}}.
\end{equation*}
\end{theorem}
Using Theorem~\ref{thmN}, Theorem~\ref{thmP}, one can show the $\gamma$-positivity of the polynomials $N(A_{n-1},x)$ and $P(x;\{1^2,2^2,\ldots,n^2\})$
via the change of grammars~\cite{Ma1902}. We omit the proofs for simplicity.
\subsection{Descent polynomials of Legendre-Stirling permutations}\label{sec0304}
\hspace*{\parindent}

The Legendre-Stirling numbers $\LS(n,k)$ were introduced in~\cite{Everitt02} as the coefficients in the integral
Lagrangian symmetric powers of the classical Legendre second-order differential expression.
The numbers $\LS(n,k)$ can be defined as follows:
$$x^n=\sum_{k=0}^n\LS(n,k){\langle{x}\rangle}_k,$$
where ${\langle{x}\rangle}_k=\prod_{i=0}^{k-1}(x-i(i+1))$ for $k\geqslant 1$ and ${\langle{x}\rangle}_0:=1$.
These numbers satisfy the recursion
$$\LS(n,k)=\LS(n-1,k-1)+k(k+1)\LS(n-1,k),$$
with the initial conditions $\LS(0,0)=1$ and $\LS(0,k)=0$ for $k\geqslant 1$.

Consider
the polynomial $L_k(x)$ defined by
$$\sum_{n=0}^\infty \LS(n+k,n)x^n=\frac{L_k(x)}{(1-x)^{3k+1}}.$$
Egge~\cite{Egge10} found that $L_k(x)$ is the descent polynomial of Legendre-Stirling permutations of order $k$.
In particular, $L_0(x)=1,~L_1(x)=2x,~L_2(x)=4x+24x^2+12x^3$.
It follows from~\cite[Theorem~4.2]{Egge10} that
\begin{equation}\label{Lnx}
\left(\frac{x^3}{1-x}\frac{\mathrm{d}^2}{\mathrm{d}x^2}\right)^n\frac{1}{1-x}=x^{3n+1}\frac{L_n(1/x)}{(1-x)^{3n+1}}.
\end{equation}

By applying the differential operator method, we set
$$T_1=\frac{\mathrm{d}}{\mathrm{d}x},~T_2=\frac{x^3}{1-x}\frac{\mathrm{d}}{\mathrm{d}x},~a=\frac{1}{1-x},~b=\frac{x}{1-x}.$$
So we have $$T_1(a)=T_1(b)=a^2,~T_2(a)=T_2(b)=b^3.$$
From~\eqref{Lnx}, we can now present the following result.
\begin{theorem}\label{thmP}
Let $G_1=\{a\rightarrow a^2,~b\rightarrow a^2\}$ and $G_2=\{a\rightarrow b^3,~b\rightarrow b^3\}$.
Then we have
\begin{equation*}
\left(D_{G_2}D_{G_1}\right)^na=\left(D_{G_2}D_{G_1}\right)^nb=b^{3n+1}L_n\left(\frac{a}{b}\right)~{\text{for $n\geqslant 1$}}.
\end{equation*}
\end{theorem}
\section{Normal ordered grammars}\label{section04}
Motivated by Proposition~\ref{prop-3}, we introduce the following definition.
\begin{definition}
Assume that $u:=u(x,y),v:=(x,y)$ and $w:=w(x,y)$ are given functions. For the context-free grammar $G=\left\{x\rightarrow u(x,y),~y\rightarrow v(x,y)\right\}$,
the powers of $w(x,y)D_G$ can be expressed in the {\it normal ordered form} as
$$\left(w(x,y)D_G\right)^n=\sum_{k=0}^n \xi_{n,k}(x,y)w^k(x,y)D_G^k.$$
\end{definition}

In the following, we only investigate the normal ordered
grammars associated with the Eulerian polynomials. In the same way, one can study the other grammars.

It is clear that Proposition~\ref{grammar03} can be restated as
\begin{equation}\label{xDG701}
(xD_{G'})^n(x)=(xD_{G'})^n(y)=y^{n+1}A_n\left(\frac{x}{y}\right),~{\text{where $G'=\{x\rightarrow y,~y\rightarrow y\}$}};
\end{equation}
\begin{equation}\label{xDG702}
(xyD_{G''})^n(x)=(xyD_{G''})^n(y)=y^{n+1}A_n\left(\frac{x}{y}\right),~{\text{where $G''=\{x\rightarrow 1,~y\rightarrow 1\}$}}.
\end{equation}
In order to investigate the powers of $xD_{G'}$ and $xD_{G''}$, we need to introduce some definitions.
The {\it degree} of a vertex in a tree is referred to the number of its children. We say that $T$ is
a {\it planted binary (resp.~full binary) increasing plane tree} on $[n]$ if it is a binary (resp.~full binary)
tree with $n$ (resp.~$n+1$) unlabeled leaves and $n$ labeled internal vertices, and satisfying the
following conditions (see Figures~\ref{Fig01} and~\ref{Fig03-xy} for examples, where we give every right leaf a weight $y$, and each of the other leaves a weight $x$):
\begin{itemize}
  \item [$(i)$] Internal vertices are labeled by $1,2,\ldots,n$. The node labelled $1$ is distinguished as the root and it has only one child (resp.~it also has two children);
 \item [$(ii)$] Except (resp. Besides) the root, each internal node has exactly two ordered children, which are referred to as a left child and a right child;
  \item [$(iii)$] For each $2\leqslant i\leqslant n$, the labels of the internal nodes in the unique
path from the root to the internal node labelled $i$ form an increasing sequence.
\end{itemize}

\tikzset{
  solid node/.style={circle,draw,inner sep=1.2,fill=black},
  hollow node/.style={circle,draw,inner sep=1.2},
  level distance = 0.6 cm,
  level 1/.style = {sibling distance = 0.8cm},
  level 2/.style = {sibling distance = 0.6cm},
  level 3/.style = {sibling distance = 0.6cm},
  every level 0 node/.style={draw,hollow node},
  every level 1 node/.style={draw,solid node},
  every level 2 node/.style={draw,solid node},
  every level 3 node/.style={draw,solid node}
}

\begin{figure}[ht!]
\begin{center}
\hspace*{\stretch{1}}
\begin{tikzpicture}
\Tree [.\node[label={1}]{};
        \edge; [.\node[label=left:{2}]{};
            \edge; [.\node[label=left:{3}]{};
                \edge; [.\node [label=below:{$x$}] {}; ]
                \edge; [.\node [label=below:{$y$}] {}; ]]
            \edge; [.\node [label=below:{$y$}] {}; ]]
        ]
\end{tikzpicture};\hspace*{\stretch{1}}
\begin{tikzpicture}
\Tree [.\node[label={1}]{};
        \edge; [.\node[label=left:{2}]{};
        	\edge; [.\node [label=below:{$x$}] {}; ]
            \edge; [.\node[label=right:{3}]{};
                \edge; [.\node [label=below:{$x$}] {}; ]
                \edge; [.\node [label=below:{$y$}] {}; ]]]
        ]
\end{tikzpicture}\hspace*{\stretch{1}}
\end{center}
\caption{The planted binary
increasing plane trees on $[3]$ encoded by $xy^2{D_{G'}}$ and $x^2yD_{G'}$, respectively .}
\label{Fig01}
\end{figure}
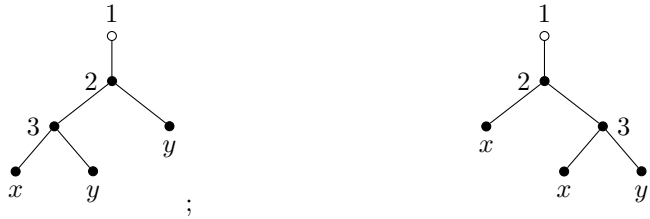
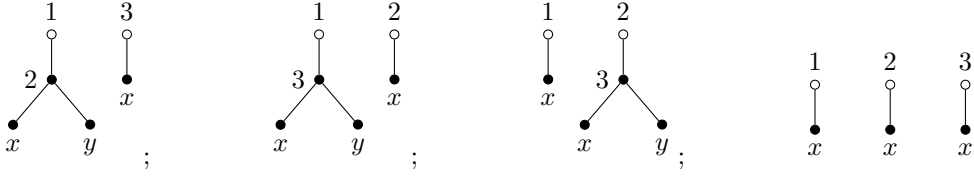
\begin{figure}[ht!]
\begin{center}
\hspace*{\stretch{1}}
\begin{tikzpicture}
    \Tree [.\node[label={1}]{};
        \edge; [.\node[label=left:{2}]{};
        	\edge; [.\node [label=below:{$x$}] {}; ]
        	\edge; [.\node [label=below:{$y$}] {}; ]]
        ]

  \begin{scope}[xshift=1cm]
		\Tree [.\node[label={3}]{};
        	\edge; [.\node[label=below:{$x$}]{};]
        ]
  \end{scope}
\end{tikzpicture};\hspace*{\stretch{1}}
\begin{tikzpicture}
    \Tree [.\node[label={1}]{};
        \edge; [.\node[label=left:{3}]{};
        	\edge; [.\node [label=below:{$x$}] {}; ]
        	\edge; [.\node [label=below:{$y$}] {}; ]]
        ]

  \begin{scope}[xshift=1cm]
    \Tree [.\node[label={2}]{};
        \edge; [.\node[label=below:{$x$}]{};]
        ]
  \end{scope}
\end{tikzpicture};\hspace*{\stretch{1}}
\begin{tikzpicture}
    \Tree [.\node[label={1}]{};
        \edge; [.\node[label=below:{$x$}]{};]
        ]

  \begin{scope}[xshift=1cm]
        \Tree [.\node[label={2}]{};
        	\edge; [.\node[label=left:{3}]{};
        		\edge; [.\node [label=below:{$x$}] {}; ]
        		\edge; [.\node [label=below:{$y$}] {}; ]]
        ]
  \end{scope}
\end{tikzpicture};\hspace*{\stretch{1}}
\begin{tikzpicture}
    \Tree [.\node[label={1}]{};
        \edge; [.\node[label=below:{$x$}]{};]
        ]

  \begin{scope}[xshift=1cm]
        \Tree [.\node[label={2}]{};
        	\edge; [.\node[label=below:{$x$}]{};]
        ]
  \end{scope}

  \begin{scope}[xshift=2cm]
        \Tree [.\node[label={3}]{};
        	\edge; [.\node[label=below:{$x$}]{};]
        ]
  \end{scope}
\end{tikzpicture}\hspace*{\stretch{1}}
\end{center}
\caption{Three 2-forests on $[3]$ encoded by $x^2yD_{G'}^2$, and the 3-forest on $[3]$ encoded by $x^3D_{G'}^3$.}
\label{Fig02}
\end{figure}

\begin{definition}
We say that $F$ is a {\it binary (resp.~full binary) $k$-forest} on $[n]$ if it has $k$ connected components, each connected component is a
planted binary (resp.~full binary) increasing plane tree, the labels of the roots are increasing from left to right and the labels of the $k$-forest form a partition of $[n]$.
\end{definition}
\begin{theorem}\label{Eulerian-tree}
Let $G'=\{x\rightarrow y, y\rightarrow y\}$.
For any $n\geqslant 1$, one has
\begin{equation}\label{aDGthm}
(xD_{G'})^n=\sum_{k=1}^n\sum_{\ell=k}^nA_{n,k,\ell}x^\ell y^{n-\ell}D_{G'}^k,
\end{equation}
where the coefficients $A_{n,k,\ell}$ satisfy the recurrence relation
\begin{equation}\label{Anki-recu}
A_{n+1,k,\ell}=\ell A_{n,k,\ell}+(n-\ell+1)A_{n,k,\ell-1}+A_{n,k-1,\ell-1},
\end{equation}
with the initial conditions $A_{1,1,1}=1$ and $A_{1,k,\ell}=0$ if $(k,\ell)\neq (1,1)$.
The coefficient $A_{n,k,\ell}$ counts binary $k$-forests on $[n]$ with $n-\ell$ right leaves.
\end{theorem}
\begin{proof}
(A) The first few $(xD_{G'})^n$ are given as follows:
\begin{align*}
(xD_{G'})^2&=xyD_{G'}+x^2D_{G'}^2,\\
(xD_{G'})^3&=(xy^2+x^2y)D_{G'}+3x^2yD_{G'}^2+x^3D_{G'}^3,\\
(xD_{G'})^4&=(xy^3+4x^2y^2+x^3y)D_{G'}+(7x^2y^2+4x^3y)D_{G'}^2+6x^3yD_{G'}^3+x^4D_{G'}^4.
\end{align*}
Thus the expansion~\eqref{aDGthm} holds for $n\leqslant 4$. Assume that it holds for $n$.
Since $$(xD_{G'})^{n+1}=xD_{G'}\left(xD_{G'}\right)^n=xD_{G'}\left(\sum_{k=1}^n\sum_{\ell=k}^nA_{n,k,\ell}x^\ell y^{n-\ell}D_{G'}^k\right),$$
it follows that
 \begin{equation}\label{xDGN}
(xD_{G'})^{n+1}=\sum_{k=1}^n\sum_{\ell=k}^nA_{n,k,\ell}\left[\left(\ell x^\ell y^{n-\ell+1}+(n-\ell)x^{\ell+1}y^{n-\ell}\right)D_{G'}^k+x^{\ell+1}y^{n-\ell}D_{G'}^{k+1}\right].
\end{equation}
Extracting the coefficient of $x^\ell y^{n-\ell+1}D_{G'}^k$ on both sides leads to the recursion~\eqref{Anki-recu}.

(B) Let $F$ be a binary $k$-forest.
We first give a labeling of $F$ as follows. Label each planted binary increasing plane tree by $D_{G'}$, a right leaf by $y$, and all the other leaves are labeled by $x$.
The weight of $F$ is defined to be the product of the labels of all trees in $F$. See Figure~\ref{Fig02} for illustrations. Assume that the weight of $F$ is $x^\ell y^{n-\ell}D_{G'}^k$.
Let us examine how to generate a forest $F'$ on $[n+1]$ by adding the vertex $n+1$ to $F$.
We have the following three possibilities:
\begin{itemize}
  \item [$c_1$:] When the vertex $n+1$ is attached to a leaf with label $x$, then $n+1$ becomes a internal node with two children. The weight of $F'$ is $x^{\ell}y^{n-\ell+1}D_{G'}^{k}$;
  \item [$c_2$:] When the vertex $n+1$ is attached to a leaf with label $y$, then $n+1$ becomes a internal node with two children. The weight of $F'$ is $x^{\ell+1}y^{n-\ell}D_{G'}^{k}$;
  \item [$c_3$:] If the vertex $n+1$ is added as a new root, then $F'$ becomes a binary $(k+1)$-forest and the child of $n+1$ has a label $x$.
  The weight of $F'$ is given by $x^{\ell+1}y^{n-\ell}D_{G'}^{k+1}$.
\end{itemize}
As each case corresponds to a term in the right of~\eqref{xDGN}, then $(xD_{G'})^{n+1}$
equals the sum of the weights of all binary $k$-forests on $[n+1]$, where $1\leqslant k\leqslant n+1$. This completes the proof.
\end{proof}

Comparing~\eqref{Anki-recu} with~\eqref{Eulerian02}, we see that $A_{n+1,1,\ell}=\Eulerian{n}{\ell}$.
Let $$A_n(x,y,z)=\sum_{k=1}^n\sum_{\ell=k}^nA_{n,k,\ell}x^\ell y^{n-\ell}z^k.$$
 Multiplying both sides of~\eqref{Anki-recu} by $x^\ell y^{n+1-\ell}z^k$ and summing over all $\ell$ and $k$, we get
 $$A_{n+1}(x,y,z)=x(n+z)A_n(x,y,z)+x(y-x)\frac{\partial}{\partial x}A_n(x,y,z),~A_0(x,y,z)=1.$$
By~\eqref{Eulerian01}, we find that $A_n(x,1,1)=A_n(x)$, where $A_n(x)$ is the Eulerian polynomial. Note that the sum of exponents of $x$ and $y$ equals $n$ in
a general term $x^\ell y^{n-\ell}z^k$.
By induction, it is easy to verify that $yA_n(1,y,1)=A_n(y)$.
Using~\eqref{Anki-recu}, we notice that $A_{n,k,k-1}=0$ and so $A_{n+1,k,k}=kA_{n,k,k}+A_{n,k-1,k-1}$.
Thus $A_{n,k,k}$ satisfies the same
recurrence and initial conditions as $\Stirling{n}{k}$.
In conclusion, we have the following result.
\begin{corollary}\label{Dual-result}
For $n\geqslant 1$, we have
\begin{align*}
&\sum_{k=1}^nA_{n,k,k}z^k=\sum_{k=1}^n\Stirling{n}{k}z^k,\\
&A_n(1,1,z)=z(z+1)\cdots(z+n-1)=\sum_{k=1}^n\stirling{n}{k}z^k,\\
&A_n(x)=A_n(x,1,1)=xA_n(1,x,1)=\frac{\partial}{\partial z}A_{n+1}(x,y,z)|_{y=1,z=0}=\sum_{\ell=1}^n\Eulerian{n}{\ell}x^\ell.
\end{align*}
\end{corollary}

\begin{figure}[ht!]
\begin{center}
\hspace*{\stretch{1}}
\begin{tikzpicture}
\Tree [.\node[label=left:{1}]{};
            \edge; [.\node[label=left:{2}]{};
                \edge; [.\node [label=below:{$x$}] {}; ]
                \edge; [.\node [label=below:{$y$}] {}; ]]
            \edge; [.\node [label=below:{$y$}] {}; ]]
        ]
\end{tikzpicture};\hspace*{\stretch{1}}
\begin{tikzpicture}
\Tree [.\node[label=left:{1}]{};
        	\edge; [.\node [label=below:{$x$}] {}; ]
            \edge; [.\node[label=right:{2}]{};
                \edge; [.\node [label=below:{$x$}] {}; ]
                \edge; [.\node [label=below:{$y$}] {}; ]]]
        ]
\end{tikzpicture}\hspace*{\stretch{1}}
\end{center}
\caption{The planted full binary
increasing plane trees on $[2]$ encoded by $xy^2{D_{G''}}$ and $x^2yD_{G''}$, respectively .}
\label{Fig03-xy}
\end{figure}
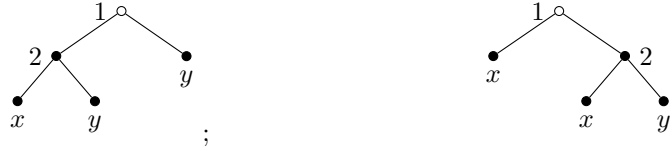

As a variant of Theorem~\ref{Eulerian-tree}, we now present the following result.
\begin{theorem}\label{Eulerian-fulltree}
Let $G''=\{x\rightarrow 1, y\rightarrow 1\}$.
For any $n\geqslant 1$, we have
\begin{equation}\label{full-xy}
(xyD_{G''})^n=\sum_{k=1}^n\sum_{\ell=k}^na_{n,k,\ell}x^\ell y^{n+k-\ell}D_{G''}^k,
\end{equation}
where the coefficients $a_{n,k,\ell}$ satisfy the recurrence relation
\begin{equation}\label{anki-recu}
a_{n+1,k,\ell}=\ell a_{n,k,\ell}+(n+k-\ell+1)a_{n,k,\ell-1}+a_{n,k-1,\ell-1},
\end{equation}
with the initial conditions $a_{1,1,1}=1$ and $a_{1,k,\ell}=0$ if $(k,\ell)\neq (1,1)$.
The coefficient $a_{n,k,\ell}$ counts full binary $k$-forests on $[n]$ with $\ell$ left leaves.
Moreover, we have
\begin{equation}\label{xyG3}
(xyD_{G''})^n=\sum_{k=1}^n\sum_{\ell=k}^{\lrf{(n+k)/2}}\gamma(n,k,\ell)(xy)^\ell (x+y)^{n+k-2\ell}D_{G''}^k,
\end{equation}
where the coefficients $\gamma(n,k,\ell)$ satisfy the recursion
\begin{equation}\label{xyGmma}
\gamma(n+1,k,\ell)=\ell\gamma(n,k,\ell)+2(n+k-2\ell+2)\gamma(n,k,\ell-1)+\gamma(n,k-1,\ell-1),
\end{equation}
with the initial conditions $\gamma(1,1,1)=1$ and $\gamma(1,k,\ell)=0$ for all $(k,\ell)\neq (1,1)$.
\end{theorem}
\begin{proof}
(A) The first few $(xyD_{G''})^n$ are given as follows:
\begin{align*}
(xyD_{G''})^2&=(xy^2+x^2y)D_{G''}+x^2y^2D_{G''}^2,\\
(xyD_{G''})^3&=(xy^3+4x^2y^2+x^3y)D_{G''}+(3x^2y^3+3x^3y^2)D_{G''}^2+x^3y^3D_{G''}^3,\\
(xyD_{G''})^4&=(xy^4+11x^2y^3+11x^3y^2+x^4y)D_{G''}+(7x^2y^4+22x^3y^3+7x^4y^2)D_{G''}^2+\\
&(6x^3y^4+6x^4y^3)D_{G''}^3+x^4y^4D_{G''}^4.
\end{align*}
Thus~\eqref{full-xy} holds for $n\leqslant 4$. Assume that the expansion holds for $n$.
Then we have
\begin{align*}
&(xyD_{G''})^{n+1}\\
&=xyD_{G''}\left(\sum_{k=1}^n\sum_{\ell=k}^na_{n,k,\ell}x^\ell y^{n+k-\ell}D_{G''}^k\right)\\
&=\sum_{k=1}^n\sum_{\ell=k}^na_{n,k,\ell}\left[\left(\ell x^\ell y^{n+k-\ell+1}+(n+k-\ell)x^{\ell+1}y^{n+k-\ell}\right)D_{G''}^k+x^{\ell+1}y^{n+k-\ell+1}D_{G''}^{k+1}\right].
\end{align*}
Extracting the coefficient of $x^\ell y^{n+k-\ell+1}D_{G''}^k$ on both sides leads to the recursion~\eqref{anki-recu}.

(B) Let $F$ be a full binary $k$-forest.
We first give a labeling of $F$ as follows. Label each planted full binary increasing plane tree by $D_{G''}$, a left leaf by $x$ and a right leaf by $y$.
The weight of $F$ is defined to be the product of the labels of all trees in $F$. See Figure~\ref{Fig03-xy} for illustrations. Assume that the weight of $F$ is $x^\ell y^{n+k-\ell}D_{G''}^k$.
Let us examine how to generate a forest $F'$ on $[n+1]$ by adding the vertex $n+1$ to $F$.
We have the following three possibilities:
\begin{itemize}
  \item [$c_1$:] When the vertex $n+1$ is attached to a leaf with label $x$, then $n+1$ becomes a internal node with two children.
  The weight of $F'$ is $x^{\ell}y^{n+k-\ell+1}D_{G''}^{k}$;
  \item [$c_2$:] When the vertex $n+1$ is attached to a leaf with label $y$, then $n+1$ becomes a internal node with two children.
  The weight of $F'$ is $x^{\ell+1}y^{n+k-\ell}D_{G''}^{k}$;
  \item [$c_3$:] If the vertex $n+1$ is added as a new root, then $F'$ becomes a full binary $(k+1)$-forest, the left child of $n+1$ has a label $x$,
  while the right child of $n+1$ has a label $y$.
  The weight of $F'$ is given by $x^{\ell+1}y^{n+k-\ell+1}D_{G''}^{k+1}$.
\end{itemize}
The above three cases exhaust all the possibilities. Thus $(xyD_{G''})^{n+1}$
equals the sum of the weights of all full binary $k$-forests on $[n+1]$, where $1\leqslant k\leqslant n+1$.

(C) We now consider a change of the grammar $G''$. Setting $u=xy$ and $v=x+y$, we get
$$D_{G''}(u)=D_{G''}(xy)=v,~D_{G''}(v)=D_{G''}(x+y)=2.$$
Let $G'''=\{u\rightarrow v,~v\rightarrow 2\}$. Then we have
$\left(xyD_{G''}\right)^n=\left(uD_{G'''}\right)^n$.
Note that
$$\left(uD_{G'''}\right)^2=uvD_{G'''}+u^2D_{G'''}^2,~\left(uD_{G'''}\right)^3=(uv^2+2u^2)D_{G'''}+3u^2vD_{G'''}^2+u^3D_{G'''}^3.$$
By induction, it is easy to check that
\begin{equation*}
\left(uD_{G'''}\right)^n=\sum_{k=1}^n\sum_{\ell=k}^{\lrf{(n+k)/2}}\gamma(n,k,\ell)u^\ell v^{n+k-2\ell}D_{G'''}^k,
\end{equation*}
where the coefficients $\gamma(n,k,\ell)$ satisfy the recursion~\ref{xyGmma}.
Then upon taking $u=xy$ and $v=x+y$, we get~\eqref{xyG3}.
This completes the proof.
\end{proof}

Comparing~\eqref{anki-recu} with~\eqref{Eulerian02}, we notice that $a_{n,1,\ell}=\Eulerian{n}{\ell}$.
Define $$a_n(x,y,z)=\sum_{k=1}^n\sum_{\ell=k}^na_{n,k,\ell}x^\ell y^{n+k-\ell}z^k,~a_0(x,y,z)=1.$$
Multiplying both sides of~\eqref{anki-recu} by $x^\ell y^{n+k-\ell+1}z^k$ and summing over all $\ell$ and $k$, we obtain
$$a_{n+1}(x,y,z)=x(n+yz)a_n(x,y,z)+x(y-x)\frac{\partial}{\partial x}a_n(x,y,z)+xz\frac{\partial}{\partial z}a_n(x,y,z).$$
In particular,
$$a_{n+1}(1,1,z)=(n+z)a_n(1,1,z)+z\frac{\mathrm{d}}{\mathrm{d} z}a_n(1,1,z),~a_0(1,1,z)=1.$$
Let $a_n(1,1,z)=\sum_{k=1}^nL(n,k)z^k$. It follows that
$L(n+1,k)=(n+k)L(n,k)+L(n,k-1)$,
from which we notice that $L(n,k)$ is the (signless) {\it Lah number}. Explicitly, $$L(n,k)=\binom{n-1}{k-1}\frac{n!}{k!}.$$
\begin{corollary}
For $n\geqslant 1$,
we have $$a_n(1,1,z)=\sum_{k=1}^n\binom{n-1}{k-1}\frac{n!}{k!}z^k.$$
\end{corollary}

A partition of $[n]$ into {\it lists} is a set
partition of $[n]$ for which the elements of each block are linearly ordered.
It is well known that $L(n,k)$ counts set partitions of $[n]$ into $k$ lists (see~\cite[A008297]{Sloane}).
We always assume that each list is prepended and appended by $0$. Given a list $\sigma_1\sigma_2\cdots\sigma_i$.
We identify it with the word $0\sigma_1\sigma_2\cdots\sigma_i0$.
We say that an index $p\in \{0,1,2,\ldots,i-1\}$ is an {\it ascent} if $\sigma_p<\sigma_{p+1}$, and $q\in \{1,2,\ldots,i\}$ is a {\it descent} if $\sigma_p>\sigma_{p+1}$, where we set $\sigma_0=\sigma_{i+1}=0$.
Let $F$ be a full binary $k$-forest. Following~\cite[p.~51]{Stanley11}, a bijection from full binary $k$-forests to set partitions with $k$ lists can be given as follows: Read the internal vertices of trees (from left to right) of $F$ in symmetric order, i.e.,
read the labels of the left subtree (in symmetric order, recursively),
then the label of the root, and then the labels of the right subtree. By this correspondence, we get the following result.
\begin{corollary}
Let $a_{n,k,\ell}$ be defined by~\eqref{full-xy}. Then $a_{n,k,\ell}$ is the number of set partitions of $[n]$ into $k$ lists with $\ell$ ascents and $n+k-\ell$ descents.
\end{corollary}

For a permutation $\pi\in\msn$ with $\pi(0)=\pi(n+1)=0$, we say that the entry $\pi(i)$
\begin{itemize}
  \item is a {\it valley} if $\pi(i-1)>\pi(i)<\pi(i+1)$;
  \item is a {\it double descent} if $\pi(i-1)>\pi(i)>\pi(i+1)$.
\end{itemize}
Let $\val(\pi)$ (resp.~$\dd(\pi)$) denote the number of valleys (resp.~double descents) in $\pi$.
Define $$\gamma(n,\ell)=\#\{\pi\in\msn: \val(\pi)=\ell,~\dd(\pi)=0\}.$$
A classical result of Foata-Sch\"utzenberger~\cite{Foata73} states that the Eulerian polynomials
have the following $\gamma$-expansion:
$$A_n(x)=x\sum_{\ell=0}^{\lrf{(n-1)/2}}\gamma(n,\ell)x^\ell(1+x)^{n-1-2\ell}.$$
Br\"and\'{e}n~\cite{Branden08} reproved this expansion by introducing the modified Foata-Strehl action.
Let $\mathcal{S}(n,k)$ be the set of partitions of $[n]$ into $k$ lists.
Applying the modified Foata-Strehl action on each list of an element in $\mathcal{S}(n,k)$, we find the following result, and omit the proof for simplicity.
\begin{corollary}
For $n\geqslant 1$, the polynomials $a_n(x,y,z)$ is partial $\gamma$-positive, i.e.,
\begin{align*}
\sum_{k=1}^n\sum_{\ell=k}^na_{n,k,\ell}x^\ell y^{n+k-\ell}z^k&=\sum_{k=1}^nz^k\sum_{\ell=k}^{\lrf{(n+k)/2}}\gamma(n,k,\ell)(xy)^\ell (x+y)^{n+k-2\ell}\\
&=\sum_{k=1}^n(xyz)^k\sum_{i=0}^{\lrf{(n-k)/2}}\gamma(n,k,k+i)(xy)^i (x+y)^{n-k-2i},
\end{align*}
where $\gamma(n,k,k+i)$ counts partitions of $[n]$ into $k$ lists with $i$ valleys and with no double descents.
\end{corollary}
\section{Box sorting algorithm and Young tableaux}\label{sec05}
Rota~\cite{Rota92} once said ``I will tell you shamelessly what my bottom line is: It is placing
balls into boxes". As discussed before, $D=\frac{\mathrm{d}}{\mathrm{d}x}$ and $c:=c(x)$.
In order to study the powers of $cD$, we shall introduce the box sorting algorithm.

An {\it ordered weak set partition} of $[n]$ is a list of pairwise disjoint subsets (maybe empty) of $[n]$
such that the union of these subsets is $[n]$. These subsets are called the {\it parts} of the partition.
A {\it weak composition} $\alpha$ of an integer $n$, denoted by $\alpha\models n$, with $m$ parts is a way of writing
$n$ as the sum of any sequence $\alpha=(\alpha_1,\alpha_2,\ldots,\alpha_m)$ of nonnegative integers.
Given $\alpha\models n$. The {\it Young weak composition diagram} of $\alpha$, also denoted by $\alpha$,
is the left-justified array of $n$ boxes
with $\alpha_i$ boxes in the $i$-th row. We follow the French convention, which
means that we number the rows from bottom to top, and the columns from left to
right. The box in the $i$-th row and $j$-th column is denoted by the pair $(i,j)$.
A {\it Young weak composition tableau} ($\YWCT$, for short) of $\alpha$ is obtained by placing the integers $\{1,2,\ldots,n\}$ into $n$ boxes of the diagram
such that each of those integers is used exactly once.
We will often identify an ordered weak set partition with the corresponding $\YWCT$.
It should be noted that there may be some empty boxes in $\YWCT$. In the following discussion,
we always put a special column of $n+1$ boxes at the left of $\YWCT$ or $\SYT$, and labelled by $0,1,2,\ldots, n$ from bottom to top, see Figure~\ref{fig3} and Table~\ref{tab:dummy-1} for instances.
\begin{figure}[!ht]
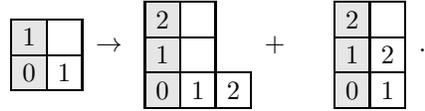
\label{fig3}
\renewcommand{\arraystretch}{2}
\begin{center}
\begin{tabular}{c}
    \begin{ytableau}
    *(gray!20) 1& \\
    *(gray!20) 0 & 1 \\
    \end{ytableau}
\end{tabular}$\rightarrow$
\begin{tabular}{c}
    \begin{ytableau}
    *(gray!20) 2& \\
    *(gray!20) 1& \\
    *(gray!20) 0 & 1&2 \\
    \end{ytableau}\\
\end{tabular}+~~
~~\quad \begin{tabular}{c}
    \begin{ytableau}
    *(gray!20) 2& \\
    *(gray!20) 1& 2\\
    *(gray!20) 0 & 1 \\
    \end{ytableau}
\end{tabular}.
\end{center}
\caption{An illustration of the change of weights: $cc_1\rightarrow c^2c_2+cc_1^2$}\label{fig3}
\end{figure}

The following label schema is fundamental.
\begin{LSp}
Let $p$ be an ordered weak set partition of $[n]$.
We give a labeling of $p$ as follows.
Label the $i$-th subset by the subscript $c_{(i-1)}$,
label a subset with $i$ elements by a superscript $c_i$, where $i\geqslant 1$. Moreover, if the $i$-th subset is empty, we always label it by a superscript $c$.
The {\it weight} of $p$ is defined as the product of the superscript labels.
\end{LSp}

Writting $c_{(0)}=c_{(i)}=c$ and $D_{(i)}=D$ for all $i\in [n]$,
we can rewrite $(cD)^n c$ as follows:
\begin{equation}\label{cdn}
\left(c_{(n)}D_{(n)}\right)\left(c_{(n-1)}D_{(n-1)}\right) \cdots \left(c_{(2)}D_{(2)}\right) \left(c_{(1)}D_{(1)}\right) c_{(0)}.
\end{equation}
A crucial observation is that the differential operator $D_{(i)}$ in~\eqref{cdn} can only applied to $c_{(k)}$, where $0\leqslant k\leqslant i-1$.
When $n=1$, we have $(cD)c=\left(c_{(1)}D_{(1)}\right) c_{(0)}=cc_1$. When $n=2$,
we have
$(cD)^2c=\left(c_{(2)}D_{(2)}\right) \left(c_{(1)}D_{(1)}\right) c_{(0)}=cc_1^2+c^2c_2$.

Next, we introduce the box sorting algorithm, designed to transform a term in the expansion of~\eqref{cdn} into an ordered weak set partition, which can be represented by a $\YWCT$.
When a new term $\left(c_{(i)}D_{(i)}\right)$ is multiplied,
the procedure can be summarized as follows:
\begin{itemize}
  \item When applying $D_{(i)}$ to $c_{(j)}$, it corresponds to the insertion of the element $i$ into the box with the subscript $c_{(j)}$;
  \item Multiplying by $c_{(i)}$ corresponds to the opening of a new empty box $\{\}^{c}_{c_{(i)}}$.
\end{itemize}

We now provide a detailed description of the {\it box sorting algorithm}.
Start with an empty box $(\{\}^c_{c_{(0)}})$. We proceed as follows:
\begin{itemize}
  \item [$BS1$:] When $n=1$, we first insert the element $1$ to the empty box, which corresponds to the operation $D_{(1)}(c_{(0)})$.
We then open a new empty box, which corresponds to the multiplication by $c_{(1)}$. Thus we get $(\{1\}_{c_{(0)}}^{c_1},\{\}^c_{c_{(1)}})$.
  \item [$BS2$:] When $n=2$, we distinguish two cases: $(i)$ we first insert the element $2$ into the first box $\{1\}^{c_1}_{c_{(0)}}$, which corresponds to apply the operation $D_{(2)}$ to  $c_{(0)}$. We then open a new empty box, which corresponds to the multiplication by $c_{(2)}$; $(ii)$
We first insert the element $2$ into the empty box $\{\}^c_{c_{(1)}}$, which corresponds to apply the operation $D_{(2)}$ to $c_{(1)}$. We then open a new empty box, which corresponds to the multiplication by $c_{(2)}$.
Therefore, we get the following correspondences between ordered weak set partitions and their weights:
$$c^2c_2\leftrightarrow (\{1,2\}^{c_2}_{c_{(0)}},\{\}^c_{c_{(1)}},\{\}^c_{c_{(2)}}),~~cc_1^2\leftrightarrow(\{1\}^{c_1}_{c_{(0)}},\{2\}^{c_1}_{c_{(1)}},\{\}^c_{c_{(2)}}).$$
The process from $BS1$ to $BS2$ can be illustrated by Figure~\ref{fig3}.
  \item [$BS3$:] If all of the elements $[i-1]$ have already been inserted, then we consider the insertion of $i$, where $i\geqslant 3$.
Suppose that we insert the element $i$ into the $k$-th box, which has the label $\{\}^{c_\ell}_{c_{(k-1)}}$, where $1\leqslant k\leqslant i$.
Then this insertion corresponds to apply $D_{(i)}$ to $c_{(k-1)}$, and the labels of the $k$-th box become $\{\}^{c_{\ell}+1}_{c_{(k-1)}}$.  We then open a new empty box, which corresponds to the multiplication by $c_{(i)}$. When $i=3$, see Figures~\ref{fig4} and~\ref{fig5} for illustrations, where each empty box in the first column of a $\YWCT$ corresponds to an empty subset.
\end{itemize}

\begin{figure}[!ht]
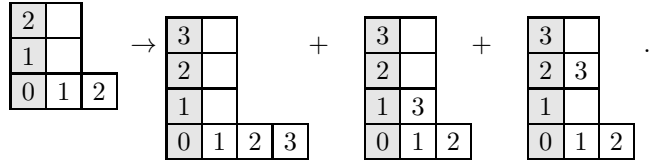
\label{fig4}
\renewcommand{\arraystretch}{2}
\begin{center}
\begin{tabular}{c}
    \begin{ytableau}
    *(gray!20) 2& \\
    *(gray!20) 1& \\
    *(gray!20) 0 & 1&2 \\
    \end{ytableau}\\
\end{tabular}$\rightarrow$
  \begin{ytableau}
    *(gray!20) 3&\\
    *(gray!20) 2&\\
    *(gray!20) 1&\\
    *(gray!20) 0 & 1 & 2 &3  \\
    \end{ytableau}+~~\quad
   \begin{ytableau}
    *(gray!20) 3&\\
    *(gray!20) 2&\\
    *(gray!20) 1 & 3\\
    *(gray!20) 0 & 1 & 2  \\
    \end{ytableau}+~~\quad
    \begin{ytableau}
    *(gray!20) 3&\\
    *(gray!20) 2 & 3\\
    *(gray!20) 1 &\\
    *(gray!20) 0 & 1 & 2  \\
    \end{ytableau}~~.
\end{center}
\caption{The insertion of $3$ into $(\{1,2\}^{c_2}_{c_{(0)}},\{\}^c_{c_{(1)}},\{\}^c_{c_{(2)}})$}\label{fig4}
\end{figure}
\begin{figure}[!ht]
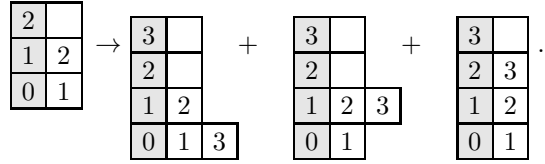
\label{fig5}
\renewcommand{\arraystretch}{2}
\begin{center}
\begin{tabular}{c}
    \begin{ytableau}
    *(gray!20) 2& \\
    *(gray!20) 1& 2\\
    *(gray!20) 0 & 1 \\
    \end{ytableau}\\
\end{tabular}$\rightarrow$
   \begin{ytableau}
    *(gray!20) 3&\\
    *(gray!20) 2&\\
    *(gray!20) 1 & 2\\
    *(gray!20) 0 & 1 & 3  \\
    \end{ytableau}+~~\quad
    \begin{ytableau}
    *(gray!20) 3&\\
    *(gray!20) 2 & \\
    *(gray!20) 1 &2&3\\
    *(gray!20) 0 & 1   \\
    \end{ytableau}+~~\quad
  \begin{ytableau}
    *(gray!20) 3&\\
    *(gray!20) 2&3\\
    *(gray!20) 1&2\\
    *(gray!20) 0 & 1   \\
    \end{ytableau}~~~.
\end{center}
\caption{The insertion of $3$ into $(\{1\}^{c_1}_{c_{(0)}},\{2\}^{c_1}_{c_{(1)}},\{\}^c_{c_{(2)}})$}\label{fig5}
\end{figure}

\begin{definition}\label{defOWP}
Let $\operatorname{OWP}_{n}$
denote the collection of ordered weak set partitions of $[n]$ into $n+1$ blocks $B_0\cup B_1\cup\cdots \cup B_n$ for which the following conditions hold:
$(a)$ $1\in B_0$; $(b)$
if $B_i$ is nonempty, then its minimum larger than $i$, where $1\leqslant i\leqslant n$.
\end{definition}

Given $p\in\operatorname{OWP}_{n}$. It is clear that the $(n+1)$-th block of $p$ must be empty.
Denote by $w_i(p)$ the number of blocks in $p$ with $i$ elements.
The {\it weight function} of $p$ is defined by
\begin{equation}\label{wp}
w(p)=\prod_{i=0}^nc_i^{w_i(p)}.
\end{equation}
By the box sorting algorithm, we immediately get the following result.
\begin{lemma}\label{Lemma-cdnc}
For $n\geqslant 1$, we have $(cD)^nc=\sum_{p\in\operatorname{OWP}_{n}}w(p)$.
\end{lemma}

Given $T\in\SYT(n)$. We define $w_i(T)$ to be the number of rows in $T$ with $i$ elements.
Let $\ell(\lambda(T))$ be the number of rows of $T$, where $\lambda$ is the shape of $T$.
Then $\ell(\lambda(T))=\sum_{i=1}^nw_i(T)$ and $n=\sum_{i=1}^niw_i(T)$.
The {\it weight function} of $T$ is defined by
\begin{equation}\label{wT}
w(T)=c^{n+1-\ell(\lambda(T))}\prod_{i=1}^nc_i^{w_i(T)}.
\end{equation}
\renewcommand{\arraystretch}{2}
\begin{center}
\begin{table}[ht!]
  \caption{$\phi^{-1}(T)=\left\{p\in\operatorname{OWP}_{3}| \phi(p)=T\right\}$, where $T\in\SYT(3)$}\label{tab:dummy-1}
  {
\begin{tabular}{ccccc|c}
$T$&&$\stackrel{\phi^{-1}}{\Longrightarrow}$&&$P$&count\\
\hline
    \begin{ytableau}
    *(gray!20) 3\\
    *(gray!20) 2\\
    *(gray!20) 1\\
    *(gray!20) 0 & 1 & 2 &3  \\
    \end{ytableau}
    & $\Longrightarrow$
    &
    \begin{ytableau}
    *(gray!20) 3&\\
    *(gray!20) 2&\\
    *(gray!20) 1&\\
    *(gray!20) 0 & 1 & 2 &3  \\
    \end{ytableau}
    &
    & $\left( \{1, 2, 3 \}, \{ \}, \{ \},\{\} \right)$
    &1\\
\hline
    \begin{ytableau}
    *(gray!20) 3\\
    *(gray!20) 2\\
    *(gray!20) 1 & 2\\
    *(gray!20) 0 & 1 & 3  \\
    \end{ytableau}
    & $\Longrightarrow$
    &
    \begin{ytableau}
    *(gray!20) 3&\\
    *(gray!20) 2&\\
    *(gray!20) 1 & 2\\
    *(gray!20) 0 & 1 & 3  \\
    \end{ytableau}+~~
    \begin{ytableau}
    *(gray!20) 3&\\
    *(gray!20) 2&\\
    *(gray!20) 1 & 2 & 3\\
    *(gray!20) 0 & 1  \\
    \end{ytableau}
    &
    & $( \{1, 3 \}, \{ 2 \}, \{  \},\{\} ), ~ ( \{1 \}, \{ 2, 3 \}, \{  \} ,\{\})$
    &2\\
\hline
    \begin{ytableau}
    *(gray!20) 3\\
    *(gray!20) 2\\
    *(gray!20) 1 & 3\\
    *(gray!20) 0 & 1 & 2  \\
    \end{ytableau}
    & $\Longrightarrow$
    &
    \begin{ytableau}
    *(gray!20) 3&\\
    *(gray!20) 2&\\
    *(gray!20) 1 & 3\\
    *(gray!20) 0 & 1 & 2  \\
    \end{ytableau}+~~
    \begin{ytableau}
    *(gray!20) 3&\\
    *(gray!20) 2 & 3\\
    *(gray!20) 1& \\
    *(gray!20) 0 & 1 & 2  \\
    \end{ytableau}
    &
    & $\left( \{1, 2 \}, \{ 3 \}, \{  \} ,\{\}), ~ ( \{1, 2 \}, \{  \}, \{ 3 \},\{\} \right)$
    &2\\
\hline
    \begin{ytableau}
    *(gray!20) 3\\
    *(gray!20) 2 & 3\\
    *(gray!20) 1 & 2\\
    *(gray!20) 0 & 1\\
    \end{ytableau}
    & $\Longrightarrow$
    &
    \begin{ytableau}
    *(gray!20) 3&\\
    *(gray!20) 2 & 3\\
    *(gray!20) 1 & 2\\
    *(gray!20) 0 & 1\\
    \end{ytableau}
    &
    & $\left( \{1 \}, \{ 2 \}, \{ 3 \},\{\} \right)$
    &1\\
    \hline
\end{tabular}}
\end{table}
\end{center}
Let $\phi$ be the map from $\operatorname{OWP}_{n}$ to $\SYT(n)$, which is described as follows:
\begin{itemize}
  \item [$OS1$:] For $p\in\operatorname{OWP}_{n}$, let $Y$ be the corresponding $\YWCT$.
  Reorder the left-justified rows of $Y$ by their length in decreasing order from the bottom to the top, and delete all empty boxes.
  \item [$OS2$:]Rearrange the entries in each column in ascending order from the bottom to the top.
\end{itemize}
In view of~\eqref{wp} and~\eqref{wT}, we see that for any $p\in\operatorname{OWP}_{n}$, one has $\phi(p)\in\SYT(n)$ and
\begin{equation}\label{wp-SYT}
w(p)=w\left(\phi(p)\right).
\end{equation}

\begin{definition}
Given $T\in\SYT(n)$. Let
$\phi^{-1}(T)=\{p\in\operatorname{OWP}_{n}|~ \phi(p)=T\}$.
We call $\#\phi^{-1}(T)$ the {\it box sorting index} of $T$.
\end{definition}
Clearly, $\#\phi^{-1}(T)=1$ for $T\in \SYT(1)$ or $T\in\SYT(2)$.
The correspondences between $\SYT(3)$ and $\operatorname{OWP}_{3}$ are listed in Table~\ref{tab:dummy-1}.
For $T\in\SYT(n)$, let $T_i$ be the element in $\SYT(i)$ obtained from $T$ by deleting the $n-i$ elements $i+1,i+2,\ldots,n$.
We denote by $\operatorname{col}_k(T_i)$ the size of the $k$-th column of $T_i$.
\begin{theorem}\label{TSYTncol}
For $T\in\SYT(n)$, the box sorting index of $T$ can be computed as follows:
\begin{equation}\label{phiT}
\#\phi^{-1}(T)=\prod_{i=1}^n\sigma_i(T).
\end{equation}
where $\sigma_i(T)$ is defined by
\begin{equation*}
\sigma_i(T)=\left\{
  \begin{array}{ll}
   i-\operatorname{col}_1(T_i)+1, & \hbox{if $i$ is in the first column of $T$;} \\
    \operatorname{col}_{k}(T_i)-\operatorname{col}_{k+1}(T_i)+1, & \hbox{if $i$ is in the $(k+1)$-th column of $T$, where $k\geqslant 1$.}
  \end{array}
\right.
\end{equation*}
We call $\sigma_i(T)$ the box sorting index of the element $i$.
Hence
\begin{equation}\label{cdnc-Thm}
(cD)^nc=\sum_{T\in\operatorname{SYT}{(n)}}\#\phi^{-1}(T)w(T)=\sum_{T\in\operatorname{SYT}{(n)}}\left(\prod_{i=1}^n\sigma_i(T)c_i^{w_i(T)}\right)c^{n+1-\ell(\lambda(T))}.
\end{equation}
\end{theorem}
\begin{proof}
In order to prove~\eqref{phiT}, we need to count the possible positions of each entry $i$ of $p\in\phi^{-1}(T)$. We distinguish two cases:
\begin{itemize}
  \item [$(i)$] Suppose that $i$ is the $r$-th entry in the first column of $T$. Then $r=\operatorname{col}_1(T_i)$. In $T_i$, the entry $i$ is maximum. By the box sorting algorithm, we see that there are $i-(r-1)$ ways to insert $i$, and each insertion generates an element of $\operatorname{OWP}_i$, see Table~\ref{tab:dummy-1} for illustrations.
  \item [$(ii)$] Suppose that $i$ is the $r$-th entry in the $(k+1)$-th column of $T$, where $k\geqslant 1$. Then $r=\operatorname{col}_{k+1}(T_i)$. In $T_i$, the entry $i$ is maximum.
 By box sorting algorithm, we find that there are $\operatorname{col}_k(T_i)-(r-1)$ ways to insert the entry $i$, and each insertion generates an element of $\operatorname{OWP}_i$.
\end{itemize}
Continuing in this way, we eventually recover all the elements in $\operatorname{OWP}_n$.
By the multiplication principle, we get~\eqref{phiT}.
Combining~\eqref{wp-SYT},~\eqref{phiT} and Lemma~\ref{Lemma-cdnc}, we arrive at~\eqref{cdnc-Thm}, and hence the proof is complete.
\end{proof}
In Table~\ref{tab:dummy-03}, one can see the box sorting indices of standard Young tableaux in $\SYT(3)$.
\section{Applications}\label{sec06}
Combing Theorem~\ref{TSYTncol} and context-free grammars, we shall give a unified interpretations of
ten kinds of well studied polynomials, including the Ramanujan polynomials.
\subsection{Ramanujan polynomials}
\hspace*{\parindent}

The {\it Ramanujan polynomials} can be defined by
$$R_{n+1}(x)=n(1+x)R_n(x)+x^2\frac{\mathrm{d}}{\mathrm{d}x}R_n(x),~R_1(x)=1.$$
In~\cite{Zeng99}, Zeng observed the connection between the
Ramanujan polynomials and Shor's interpretation of the Cayley's formula for counting labelled trees.
Since then, there has been much work on the generalizations of Ramanujan polynomials.
Recently, Sokal~\cite{Sokal23} found the Hankel matrix $(R_{i+j}(x))_{i,j\geqslant 0}$ is coefficientwise totally positive.
It is well known that Lambert's equation $we^{-w}=y$ has an explicit solution $$w=\sum_{n\geqslant 1}n^{n-1}\frac{y^n}{n!}.$$
Differentiating $n$ times Lambert's function $w$ with respect to $y$ yields
$$w^{(n)}=\frac{\mathrm{e}^{nw}}{(1-w)^n}R_n\left(\frac{1}{1-w}\right).$$

For a rooted tree $T$ and a vertex $v$ of $T$, let $T_v$ be the subtree of $T$ rooted at $v$.
An edge $(u,v)$ of $T$, with $v$ being the child of $u$, is called an
{\it improper edge} if there exists a vertex in $T_v$ that is smaller than $u$.
Let $r(n,k)$ denotes the number of rooted trees on $[n]$ with $k$ improper edges.
Utilizing Shor's recursive procedure to construct rooted trees,
Dumont-Ramamonjisoa~\cite{Dumont9602} discovered the following result:
\begin{equation}\label{D14xy}
D_{G}^n(xy)=(xy)^{n+1}R_{n+1}(x)=(xy)^{n+1}\sum_{k=0}^nr(n+1,k)x^k,
\end{equation}
where $G=\{x\rightarrow x^3y,~y\rightarrow xy^2\}$.
The reader is referred to Chen-Yang~\cite{ChenYang21} for some recent generalizations of the Ramanujan polynomials.
In particular, they gave a proof of~\eqref{D14xy} in the language of a grammatical labeling of rooted trees.

We can now present the following result.
\begin{theorem}\label{Rnx-thm}
Let $R_n(x)$ be the Ramanujan polynomials. Then we have
$$R_{n+1}(x)=\sum_{T\in\operatorname{SYT}{(n)}}\prod_{i=1}^n\sigma_i(T)\left(\sum_{j=0}^i\binom{i}{j}j!x^{j}\right)^{w_{i}(T)},$$
where $w_i(T)$ is the number of rows in $T$ with $i$ elements and $\sigma_i(T)$ is the box sorting index of $i$.
\end{theorem}
\begin{proof}
Note that~\eqref{D14xy} can be rewritten as
$$(xyD_{G'})^n(xy)=(xy)^{n+1}R_{n+1}(x),$$
where $G'=\{x\rightarrow x^2,~y\rightarrow y\}$.
Setting $c=xy$, it is routine to check that
$$c_{i}=D_{G'}^{i}(c)=D_{G'}^{i}(xy)=xy\sum_{j=0}^i\binom{i}{j}j!x^j~{\text {for $i\geqslant 0$}}.$$
Recall that $\ell(\lambda(T))=\sum_{i\geqslant 0}w_{i}(T)$.
By~\eqref{cdnc-Thm}, we find that
\begin{align*}
x^{n+1}R_{n+1}(x)&=\sum_{T\in\operatorname{SYT}{(n)}}\left(\prod_{i=1}^n\sigma_i(T)c_i^{w_i(T)}\right)c^{n+1-\ell(\lambda(T))}{\bigg|}_{c=xy,~c_{i}=xy\sum_{j=0}^i\binom{i}{j}j!x^j,~y=1}.
\end{align*}
Thus
\begin{align*}
x^{n+1}R_{n+1}(x)&=\sum_{T\in\operatorname{SYT}{(n)}}\prod_{i=1}^n\sigma_i(T)x^{w_{i}(T)}\left(\sum_{j=0}^i\binom{i}{j}j!x^{j}\right)^{w_{i}(T)}x^{n+1-\ell(\lambda(T))}\\
&=\sum_{T\in\operatorname{SYT}{(n)}}x^{\sum_{i\geqslant 1}w_{i}(T)}\prod_{i=1}^n\sigma_i(T)\left(\sum_{j=0}^i\binom{i}{j}j!x^{j}\right)^{w_{i}(T)}x^{n+1-\ell(\lambda(T))},
\end{align*}
which yields the desired formula. This completes the proof.
\end{proof}
Setting
$\alpha_i=\sum_{j=0}^i\binom{i}{j}j!x^{j}$,
we see that $\alpha_1=1+x,~\alpha_2=1+2x+2x^2,~\alpha_3=1+3x+6x^2+6x^3$.
An illustration of Theorem~\ref{Rnx-thm} is given by Table~\ref{tab:dummy-03}.
\renewcommand{\arraystretch}{2}
\begin{center}
\begin{table}[ht!]
  \caption{The computation of $R_4(x)=6 + 18 x + 25 x^2 + 15 x^3$.}\label{tab:dummy-03}
  {
\begin{tabular}{ccccc|c}
$T$&&${\sigma_i(T)},~w_i(T)$&&&enumerator\\
\hline
    \begin{ytableau}
    *(gray!20) 3\\
    *(gray!20) 2\\
    *(gray!20) 1\\
    *(gray!20) 0 & 1 & 2 &3  \\
    \end{ytableau}
    & $\Longrightarrow$
    &
   $\substack{\sigma_1(T)=\sigma_2(T)=\sigma_3(T)=1\\ w_1(T)=w_2(T)=0,~w_3(T)=1}$
    & &
    &$1+3x+6x^2+6x^3$\\
\hline
    \begin{ytableau}
    *(gray!20) 3\\
    *(gray!20) 2\\
    *(gray!20) 1 & 2\\
    *(gray!20) 0 & 1 & 3  \\
    \end{ytableau}
    & $\Longrightarrow$
    &
  $\substack{\sigma_1(T)=\sigma_2(T)=1,~\sigma_3(T)=2\\w_1(T)=w_2(T)=1,~w_3(T)=0}$
    &
    &
    &$2(1+x)(1+2x+2x^2)$\\
\hline
    \begin{ytableau}
    *(gray!20) 3\\
    *(gray!20) 2\\
    *(gray!20) 1 & 3\\
    *(gray!20) 0 & 1 & 2  \\
    \end{ytableau}
    & $\Longrightarrow$
    &
   $\substack{\sigma_1(T)=\sigma_2(T)=1,~\sigma_3(T)=2\\w_1(T)=w_2(T)=1,~w_3(T)=0}$
    &
    &
    &$2(1+x)(1+2x+2x^2)$\\
\hline
    \begin{ytableau}
    *(gray!20) 3\\
    *(gray!20) 2 & 3\\
    *(gray!20) 1 & 2\\
    *(gray!20) 0 & 1\\
    \end{ytableau}
    & $\Longrightarrow$
    &
  $\substack{\sigma_1(T)=\sigma_2(T)=\sigma_3(T)=1\\w_1(T)=3,~w_2(T)=w_3(T)=0}$
    &
    &
    &$(1+x)^3$\\
    \hline
\end{tabular}}
\end{table}
\end{center}
\subsection{Andr\'e polynomials}
\hspace*{\parindent}

An increasing tree on $[n]$ is a rooted tree with vertex set $\{0,1,2,\ldots,n\}$ in which the labels
of the vertices are increasing along any path from the root $0$.
A {\it 0-1-2 increasing tree} is an increasing tree in which the degree of any vertex is
at most two.  Given a 0-1-2 increasing tree $T$, let $\ell(T)$ denote the number
of leaves of $T$, and $u(T)$ denote the number of vertices of $T$ with degree $1$. The
{\it Andr\'e polynomials} are defined by
$$E_n(x,y)=\sum_{T\in \men}x^{\ell(T)}y^{u(T)},$$
where $\men$ is the set of 0-1-2 increasing trees on $[n-1]$.
Here are the first few values:
\begin{equation*}
E_1(x,y)=x,~
E_2(x,y)=xy,~
E_3(x,y)=xy^2+x^2,~
E_4(x,y)=xy^3+4x^2y.
\end{equation*}

Using a grammatical labeling of 0-1-2 increasing trees,
Dumont~\cite{Dumont96} discovered that
\begin{equation}\label{DGNy}
D_G^n(y)=E_n(x,y).
\end{equation}
where $G=\{x\rightarrow xy,~y\rightarrow x\}$.
Define $$E(x,y;z)=y+\sum_{n\geqslant 1}E_n(x,y)\frac{z^n}{n!}.$$
The explicit formula of $E(x,y;z)$ was first obtained by Foata-Sch\"utzenberger~\cite{Foata73} by using a differential equation.
Dumont~\cite{Dumont96} found that $E(1,1;z)=\tan z+ \sec z$. Thus $E_n(1,1)$ is the $n$-th Euler number,
which counts alternating permutations in $\msn$, see~\cite{Stanley10} for a survey on this topic.
In~\cite{Foata01}, Foata-Han deduced the explicit formula for $E(x,1;z)$ via an identity.
In~\cite{Chen17}, Chen-Fu used the grammar $G$ to derive the explicit formula of $E(x,y;z)$.
\renewcommand{\arraystretch}{2}
\begin{center}
\begin{table}[ht!]
  \caption{The computation of $E_4(x,y)=xy^3+4x^2y$.}\label{tab:dummy-05}
  {
\begin{tabular}{ccccc|c}
$T\in\SYT(3;2)$&&${\sigma_i(T)},~w_i(T)$&&&enumerator\\
\hline
    \begin{ytableau}
    *(gray!20) 3\\
    *(gray!20) 2\\
    *(gray!20) 1 & 2\\
    *(gray!20) 0 & 1 & 3  \\
    \end{ytableau}
    &
    &
  $\substack{\sigma_1(T)=\sigma_2(T)=1,~\sigma_3(T)=2\\w_1(T)=w_2(T)=1,~w_3(T)=0}$
    &
    &
    &$2x^2y$\\
\hline
    \begin{ytableau}
    *(gray!20) 3\\
    *(gray!20) 2\\
    *(gray!20) 1 & 3\\
    *(gray!20) 0 & 1 & 2  \\
    \end{ytableau}
    &
    &
   $\substack{\sigma_1(T)=\sigma_2(T)=1,~\sigma_3(T)=2\\w_1(T)=w_2(T)=1,~w_3(T)=0}$
    &
    &
    &$2x^2y$\\
\hline
    \begin{ytableau}
    *(gray!20) 3\\
    *(gray!20) 2 & 3\\
    *(gray!20) 1 & 2\\
    *(gray!20) 0 & 1\\
    \end{ytableau}
    &
    &
  $\substack{\sigma_1(T)=\sigma_2(T)=\sigma_3(T)=1\\w_1(T)=3,~w_2(T)=w_3(T)=0}$
    &
    &
    &$xy^3$\\
    \hline
\end{tabular}}
\end{table}
\end{center}
\begin{definition}
Let $\SYT(n;k)$ be the subset of $\SYT(n)$ with at most $k$ columns.
\end{definition}

We can now present an expression of Andr\'e polynomials, see Table~\ref{tab:dummy-05} for an illustration.
\begin{theorem}\label{thmEn}
We have
$$E_{n+1}(x,y)=\sum_{T\in\operatorname{SYT}{(n;2)}}\left(\prod_{i=1}^n\sigma_i(T)\right)x^{n+1-w_1(T)-w_2(T)}y^{w_1(T)},$$
where $w_i(T)$ is the number of rows in $T$ with $i$ elements and $\sigma_i(T)$ is the box sorting index of $i$.
\end{theorem}
\begin{proof}
Let $G'=\{x\rightarrow y,~y\rightarrow 1\}$. We find that~\eqref{DGNy} can be rewritten as follows:
\begin{equation}\label{En1xy}
\left(xD_{G'}\right)^{n+1}(y)=\left(xD_{G'}\right)^{n}\left(xD_{G'}(y)\right)=\left(xD_{G'}\right)^{n}(x)=E_{n+1}(x,y).
\end{equation}
Setting $c=x$, we get $c_1=D_{G'}(c)=D_{G'}(x)=y$, $c_2=D_{G'}^2(c)=D_{G'}(y)=1$ and $c_i=0$ for all $i\geqslant 3$.
By~\eqref{cdnc-Thm}, we find that
\begin{align*}
E_{n+1}(x,y)&=\sum_{T\in\operatorname{SYT}{(n)}}\left(\prod_{i=1}^n\sigma_i(T)c_i^{w_i(T)}\right)c^{n+1-\ell(\lambda(T))}{\bigg|}_{c=x,~c_1=y,~c_2=1,~c_{i}=0~{\text{for $i\geqslant 3$}}}\\
&=\sum_{T\in\operatorname{SYT}{(n;2)}}\prod_{i=1}^n\sigma_i(T)x^{n+1-\ell(\lambda(T))}y^{w_1(T)},
\end{align*}
as desired. This completes the proof.
\end{proof}
\subsection{Interior peak and left peak polynomials}
\hspace*{\parindent}

Recently, there has been much work devoted to the combinatorial and algebraic properties of the interior peak and left peak polynomials,  see~\cite{Chen23,Gessel20,Ma23,Petersen06,Petersen15}.
Define
$$W_n(x)=\sum_{\pi\in\msn}x^{\ipk(\pi)},~L_n(x)=\sum_{\pi\in\msn}x^{\lpk(\pi)},$$
where
\begin{align*}
\ipk(\pi)&=\#\{i\in[2,n-1]:~\pi(i-1)<\pi(i)>\pi(i+1)\},\\
\lpk(\pi)&=\#\{i\in[n-1]: \pi(i-1)<\pi(i)>\pi(i+1),~\pi(0)=0\}.
\end{align*}
They satisfy the following recursions:
\begin{align*}
W_{n+1}(x)&=(nx-x+2)W_n(x)+2x(1-x)\frac{\mathrm{d}}{\mathrm{d}x}W_n(x),~W_0(x)=W_1(x)=1;\\
L_{n+1}(x)&=(nx+1)L_n(x)+2x(1-x)\frac{\mathrm{d}}{\mathrm{d}x}L_n(x),~L_0(x)=L_1(x)=1.
\end{align*}

The grammar for the interior peak and left peak polynomials was discovered in~\cite{Ma121} via the differential systems of tangent and secant functions.
Chen-Fu~\cite{Chen17} independently found the same grammar via a grammatical labeling.
In~\cite{Chen23}, they demonstrate the efficiency of the grammatical calculus in
the study of exponential generating functions of the interior peak and left peak polynomials.
Here is concerned grammar:
$G=\{x\rightarrow xy,~y\rightarrow x^2\}$.
The grammatical descriptions of $W_n(x)$ and $L_n(x)$ are respectively given as follows:
\begin{equation}\label{DG12}
D_G^n(y)=x^2y^{n-1}W_n\left(\frac{x^2}{y^2}\right),~
D_{G}^n(x)=xy^{n}L_n\left(\frac{x^2}{y^2}\right)~{\text{for $n\geqslant 1$}}.
\end{equation}
By Leibniz rule, one has $$D_G^{n+1}(y)=D_G^n(x^2)=\sum_{k=0}^n\binom{n}{k}D_G^k(x)D_G^{n-1}(x).$$
Set $L_0(x)=1$.
It follows that
\begin{equation*}\label{In}
W_{n+1}(x)=\sum_{k=0}^n\binom{n}{k}L_k(x)L_{n-k}(x).
\end{equation*}
\begin{theorem}\label{Lnx-thm}
Let $L_n(x)$ be the left peak polynomials. Then we have
$$L_n(x)=\sum_{T\in\operatorname{SYT}{(n)}}\left(\prod_{i=1}^n\sigma_i(T)\right)x^{\frac{1}{2}\left(n-\sum_{i\geqslant 0}w_{2i+1}(T)\right)}.$$
\end{theorem}
\begin{proof}
Let $G'=\{x\rightarrow y,~y\rightarrow x\}$. Note that the second expression in~\eqref{DG12} can be rewritten as
$$(xD_{G'})^n(x)=xy^{n}L_n\left(\frac{x^2}{y^2}\right).$$
Setting $c=x$, we obtain $c_{2i}=D_{G'}^{2i}(x)=x$ and $c_{2i+1}=D_{G'}^{2i+1}(x)=y$ for any $i\geqslant 0$.
By~\eqref{cdnc-Thm}, we find that
\begin{align*}
xL_n(x^2)&=\sum_{T\in\operatorname{SYT}{(n)}}\left(\prod_{i=1}^n\sigma_i(T)c_i^{w_i(T)}\right)c^{n+1-\ell(\lambda(T))}{\bigg|}_{\substack{c=x,~c_{2i}=x,\\~c_{2i+1}=y,~y=1}}\\
&=\sum_{T\in\operatorname{SYT}{(n)}}\left(\prod_{i=1}^n\sigma_i(T)c_i^{w_i(T)}\right)x^{n+1-\ell(\lambda(T))}{\bigg|}_{c_{2i}=x,~c_{2i+1}=1}\\
&=\sum_{T\in\operatorname{SYT}{(n)}}\left(\prod_{i=1}^n\sigma_i(T)\right)\left(\prod_{i=1}^{\lrf{n/2}}x^{w_{2i}(T)}\right)x^{n+1-\ell(\lambda(T))}.
\end{align*}
The desired formula follows from the fact that $\ell(\lambda(T))=\sum_{i\geqslant 0}w_{2i+1}(T)+\sum_{i\geqslant 0}w_{2i}(T)$.
\end{proof}

In~\cite[Section~2.4]{Dumont96}, Dumont derived the following grammar for increasing binary trees:
$$G=\{x\rightarrow 2xy,~y\rightarrow x\},$$
which is an analogous grammar for the Andr\'e polynomials (see~\cite{Chen2301}). Dumont noted that
\begin{equation}\label{Wnx}
D_G^n(y)=xy^{n-1}W_n\left(\frac{x}{y^2}\right),
\end{equation}
for $n\geqslant 1$, where $G=\{x\rightarrow 2xy,~y\rightarrow x\}$.
We can now give the following result. An illustration of it is given by Table~\ref{tab:dummy-ipk}.
\begin{theorem}
Let $W_n(x)$ be the interior peak polynomials. Then
$$W_{n+1}(x)=\sum_{T\in\operatorname{SYT}{(n;2)}}2^{\ell(\lambda(T))}\prod_{i=1}^n\sigma_i(T)x^{n-\ell(\lambda(T))}.$$
\end{theorem}
\begin{proof}
Let $G'=\{x\rightarrow 2y,~y\rightarrow 1\}$.
We find that~\eqref{Wnx} can be rewritten as follows:
\begin{equation}\label{En1xy}
\left(xD_{G'}\right)^{n+1}(y)=\left(xD_{G'}\right)^{n}\left(xD_{G'}(y)\right)=\left(xD_{G'}\right)^{n}(x)=xy^{n}W_{n+1}\left(\frac{x}{y^2}\right).
\end{equation}
Setting $c=x$, we get $c_1=D_{G'}(c)=D_{G'}(x)=2y$, $c_2=D_{G'}(c_1)=D_{G'}(2y)=2$ and $c_i=0$ for all $i\geqslant 3$.
By~\eqref{cdnc-Thm}, we find that
\begin{align*}
xW_{n+1}(x)&=\sum_{T\in\operatorname{SYT}{(n)}}\left(\prod_{i=1}^n\sigma_i(T)c_i^{w_i(T)}\right)c^{n+1-\ell(\lambda(T))}{\bigg|}_{\substack{c=x,~c_1=2y,~c_2=2,\\y=1,~c_{i}=0~{\text{for $i\geqslant 3$}}}}\\
&=\sum_{T\in\operatorname{SYT}{(n;2)}}\prod_{i=1}^n\sigma_i(T)x^{n+1-\ell(\lambda(T))}2^{w_1(T)+w_2(T)}.
\end{align*}
Since $\ell(\lambda(T))=w_1(T)+w_2(T)$ for $T\in\operatorname{SYT}{(n;2)}$, we immediately get the desired result.
\end{proof}

\renewcommand{\arraystretch}{2}
\begin{center}
\begin{table}[ht!]
  \caption{The computation of $W_4(x)=8+16x$.}\label{tab:dummy-ipk}
  {
\begin{tabular}{ccccc|c}
$T\in\SYT(3;2)$&&${\sigma_i(T)},~\ell(\lambda(T))$&&&enumerator\\
\hline
    \begin{ytableau}
    *(gray!20) 3\\
    *(gray!20) 2\\
    *(gray!20) 1 & 2\\
    *(gray!20) 0 & 1 & 3  \\
    \end{ytableau}
    &
    &
  $\substack{\sigma_1(T)=\sigma_2(T)=1,~\sigma_3(T)=2\\\ell(\lambda(T))=2}$
    &
    &
    &$2^2\cdot2x$\\
\hline
    \begin{ytableau}
    *(gray!20) 3\\
    *(gray!20) 2\\
    *(gray!20) 1 & 3\\
    *(gray!20) 0 & 1 & 2  \\
    \end{ytableau}
    &
    &
   $\substack{\sigma_1(T)=\sigma_2(T)=1,~\sigma_3(T)=2\\\ell(\lambda(T))=2}$
    &
    &
    &$2^2\cdot2x$\\
\hline
    \begin{ytableau}
    *(gray!20) 3\\
    *(gray!20) 2 & 3\\
    *(gray!20) 1 & 2\\
    *(gray!20) 0 & 1\\
    \end{ytableau}
    &
    &
  $\substack{\sigma_1(T)=\sigma_2(T)=\sigma_3(T)=1\\\ell(\lambda(T))=3}$
    &
    &
    &$2^3$\\
    \hline
\end{tabular}}
\end{table}
\end{center}
\subsection{Eulerian polynomials}
\hspace*{\parindent}

The bivariate Eulerian polynomials are defined by
\begin{equation}\label{Anxy}
A_n(x,y)=\sum_{\pi\in\msn}x^{\des(\pi)+1}y^{n-\des(\pi)}.
\end{equation}
\begin{theorem}\label{AnxSYT}
We have
\begin{equation*}\label{Anx-01}
A_n(x)=\sum_{T\in \SYT(n)}\left(\prod_{i=1}^n \sigma_i(T)\right)  ~ x^{n+1-\ell(\lambda(T))}=\sum_{T\in \SYT(n)}\left(\prod_{i=1}^n \sigma_i(T)\right)  ~ x^{\ell(\lambda(T))}.
\end{equation*}
The bivariate Eulerian polynomials have the expression
\begin{equation}\label{Anx-02}
A_{n+1}(x, y)= \sum_{T\in \SYT(n;2)}2^{w_2(T)}\prod_{i=1}^n\sigma_i(T)(xy)^{n+1-w_1(T)-w_2(T)}(x+y)^{w_1(T)},
\end{equation}
where $w_i(T)$ is the number of rows in $T$ with $i$ elements.
\end{theorem}
\begin{proof}
(A) From~\eqref{xDG701}, we see that
$A_n(x)=\left(xD_{G'}\right)^{n}(x)|_{y=1}$, where $G'=\{x\rightarrow y, y\rightarrow y\}$.
Taking $c=x$ in~\eqref{cdnc-Thm}, we get
$c_i=D_{G'}^i(c)=D_{G'}^i(x)=y$ for $i\geqslant 1$.
Combining this with~\eqref{cdnc-Thm}, we arrive at
\begin{align*}
A_n(x)&=\sum_{T\in\operatorname{SYT}{(n)}}\left(\prod_{i=1}^n\sigma_i(T)c_i^{w_i(T)}\right)c^{n+1-\ell(\lambda(T))}{\bigg|}_{c=x,~c_i=y=1}\\
&= \sum_{T\in \SYT(n)} \left(\prod_{i=1}^n \sigma_i(T) \right) ~ x^{n+1-\ell(\lambda(T))}.
\end{align*}

(B) Let $G'''=\{x \rightarrow x, y \rightarrow x\}$.
Note that Proposition~\ref{grammar03} can also be restated as
$$(yD_{G'''})^{n}(y)|_{y=1}=A_n(x).$$
Taking $c=y$ in~\eqref{cdnc-Thm}, we get $c_i=D_{G'''}^i(c)=D_{G'''}^i(y)=x$ for $i\geqslant 1$.
It follows from~\eqref{cdnc-Thm} that
\begin{align*}
A_n(x)&=\sum_{T\in\operatorname{SYT}(n)}\left(\prod_{i=1}^n\sigma_i(T)c_i^{w_i(T)}\right)c^{n+1-\ell(\lambda(T))}{\bigg|}_{c=y=1,~c_i=x}\\
&= \sum_{T\in \SYT(n)} \left(\prod_{i=1}^n \sigma_i(T) \right) ~ x^{\sum_{i=1}^nw_i(T)}\\
&= \sum_{T\in \SYT(n)} \left(\prod_{i=1}^n \sigma_i(T) \right) ~ x^{\ell(\lambda(T))}.
\end{align*}

(C) Let $G''=\{x \rightarrow 1, y \rightarrow 1\}$.
It follows from~\eqref{xDG702} that
$(xyD_{G''})^{n}(xy)=A_{n+1}(x,y)$.
Taking $c=xy$, we get $c_1=D_{G''}(c)=D_{G''}(xy)=x+y$, $c_2=D_{G''}(c_1)=2$ and $c_k=D_{G''}^k(c)=0$ for $k\geqslant 3$.
It follows from~\eqref{cdnc-Thm} that
\begin{align*}
A_{n+1}(x,y)&=\sum_{T\in\operatorname{SYT}{(n;2)}}\left(\prod_{i=1}^n\sigma_i(T)c_i^{w_i(T)}\right)c^{n+1-\ell(\lambda(T))}{\bigg|}_{c=xy,~c_1=x+y,~c_2=2}\\
&= \sum_{T\in \SYT(n;2)}2^{w_2(T)}\prod_{i=1}^n\sigma_i(T)(xy)^{n+1-\ell(\lambda(T))}(x+y)^{w_1(T)}.
\end{align*}
So we arrive at~\eqref{Anx-02}, since $\ell(\lambda(T))=w_1(T)+w_2(T)$ for
$T\in\operatorname{SYT}{(n;2)}$.
\end{proof}

\subsection{The $1/2$-Eulerian polynomials}
\hspace*{\parindent}

Following~\eqref{Ankx-def05}, the $1/2$-Eulerian polynomials $A_n^{{(2)}}(x)$ can be defined by
\begin{equation}\label{An2x-def}
\left(2x\frac{\mathrm{d}}{\mathrm{d}x}\right)^n\frac{1}{\sqrt{1-x}}=\frac{x^nA_n^{(2)}(1/x)}{(1-x)^{n+\frac{1}{2}}}.
\end{equation}
Below are the polynomials $A_n^{{(2)}}(x)$ for $n\leqslant 4$:
\begin{equation*}
A_1^{{(2)}}(x)=1,~
A_2^{{(2)}}(x)=1+2x,~
A_3^{{(2)}}(x)=1+10x+4x^2,~
A_4^{{(2)}}(x)=1+36x+60x^2+8x^3.
\end{equation*}
As a special case of~\eqref{Ankx-Explicit}, we see that
\begin{equation*}\label{An2x-Explicit}
A_n^{(2)}(x)=\sum_{i=1}^n(2i-1)!!\Stirling{n}{i}(2x-2)^{n-i}.
\end{equation*}
The polynomials $A_n^{(2)}(x)$ have several combinatorial and geometric interpretations, see~\cite{Chao19,Savage16} for details.
Here we list two well known interpretations:
$$A_n^{(2)}(x)=\sum_{\pi\in\msn}x^{\exc(\pi)}2^{n-\cyc(\pi)}=\sum_{\sigma\in\mqn}x^{\ap(\sigma)}.$$

We can now give the following result.
\begin{theorem}
Let $A_n^{{(2)}}(x)$ be the $1/2$-Eulerian polynomials, and let $W_n(x)$ be the interior peak polynomials.
Then we have
$$A_n^{{(2)}}(x)=\sum_{T\in\operatorname{SYT}{(n)}}x^{n-\ell(\lambda(T))}\prod_{i=1}^n\sigma_i(T)\left(W_i\left(\frac{1}{x}\right)\right)^{w_i(T)}.$$
\end{theorem}
\begin{proof}
Let $G'=\{x\rightarrow y^2,~y\rightarrow xy\}$.
By~\eqref{Tnak-Euler}, we find that
\begin{equation*}
\left(xD_{G'}\right)^n(x)=xy^{2n}A_n^{(2)}\left(\frac{x^2}{y^2}\right).
\end{equation*}
Setting $c=x$, it follows from~\eqref{DG12} that
$$c_i=D_{G'}^i(x)=y^2x^{i-1}W_i\left(\frac{y^2}{x^2}\right)~{\text{for $n\geqslant 1$}}.$$
Using~\eqref{cdnc-Thm}, we find that
\begin{equation*}
xA_n^{(2)}(x^2)=\sum_{T\in\operatorname{SYT}{(n)}}\left(\prod_{i=1}^n\sigma_i(T)c_i^{w_i(T)}\right)x^{n+1-\ell(\lambda(T))}{\bigg|}_{c_{i}=x^{i-1}W_i\left(\frac{1}{x^2}\right)}.
\end{equation*}
Since $$\prod_{i=1}^nx^{(i-1)w_i(T)}=x^{\sum_{i=1}^n(i-1)w_i(T)}=x^{\sum_{i=1}^niw_i(T)-\ell(\lambda(T))}=x^{n-\ell(\lambda(T))},$$
after simplifying, we obtain the desired result. This completes the proof.
\end{proof}
\subsection{Eulerian polynomials of type $B$}
\hspace*{\parindent}

The study of the type $B$ Eulerian polynomials was initiated by Brenti~\cite{Bre94}.
The types $A$ and $B$ Eulerian polynomials share several similar properties, including real zeros~\cite{Savage15}, combinatorial expansions~\cite{Bre94,Gessel20} and
geometric interpretations~\cite{Solus20}.
We now present an interpretation of the type $B$ Eulerian polynomials.
\begin{theorem}
Let $B_n(x)$ be the type $B$ Eulerian polynomials. We have
$$B_n(x)=\sum_{T\in\operatorname{SYT}{(n)}}\left(\prod_{i=1}^n\sigma_i(T)c_i^{w_i(T)}\right)x^{{\frac{1}{2}}\left(n-\ell(\lambda(T))\right)},$$
where $c_{2i-1}=4^{i-1}(1+x)$ and $c_{2i}=4^i\sqrt{x}$ for $i\geqslant 1$.
\end{theorem}
\begin{proof}
From~\eqref{DG3ab}, we see that
$xB_n(x^2)=(xyD_{G'})^n(xy){\bigg|}_{y=1}$, where $G'=\{x\rightarrow y,~y\rightarrow x\}$.
Taking $c=xy$, we notice that $$c_{2i-1}=D_{G'}^{2i-1}(c)=4^{i-1}(x^2+y^2),~c_{2i}=D_{G'}^{2i}(c)=4^ixy~{\text{for $n\geqslant 1$}}.$$
By~\eqref{cdnc-Thm}, we find that
\begin{align*}
xB_n(x^2)&=\sum_{T\in\operatorname{SYT}{(n)}}\left(\prod_{i=1}^n\sigma_i(T)c_i^{w_i(T)}\right)c^{n+1-\ell(\lambda(T))}{\bigg|}_{\substack{c_{2i-1}=4^{i-1}(x^2+y^2),~c_{2i}=4^ixy,\\y=1}}\\
&=\sum_{T\in\operatorname{SYT}{(n)}}\left(\prod_{i=1}^n\sigma_i(T)c_i^{w_i(T)}\right)x^{n+1-\ell(\lambda(T))}{\bigg|}_{c_{2i-1}=4^{i-1}(1+x^2),~c_{2i}=4^ix},
\end{align*}
which yields the desired result.
\end{proof}
\subsection{Second-order Eulerian polynomials}
\hspace*{\parindent}

Given a Stirling permutation $\sigma\in\mqn$.
The number of descents, ascents and plateaux of $\sigma$ are respectively defined as follows:
\begin{align*}
\des(\pi)&=\#\{i\in[2n]:~\sigma_i>\sigma_{i+1},~\sigma_{2n+1}=0\},\\
\asc(\pi)&=\#\{i\in[0,2n-1]:~\sigma_i<\sigma_{i+1},~\sigma_0=0\},\\
\plat(\pi)&=\#\{i\in[2n-1]:~ \sigma_i=\sigma_{i+1}\}.
\end{align*}
The {\it trivariate second-order Eulerian polynomials} are defined by
\begin{equation}\label{Cnxyz011}
C_n(x,y,z)=\sum_{\sigma\in\mqn}x^{\asc{(\sigma)}}y^{\des(\sigma)}z^{\plat{(\sigma)}}.
\end{equation}
In~\cite[p.~317]{Dumont80},
Dumont found that
\begin{equation}\label{Dumont80}
C_{n+1}(x,y,z)=xyz\left(\frac{\partial}{\partial x}+\frac{\partial}{\partial y}+\frac{\partial}{\partial z}\right)C_n(x,y,z),
\end{equation}
which implies that $C_n(x,y,z)$ is symmetric in the variables $x,y$ and $z$.
By~\eqref{Dumont80}, it is clear that
\begin{equation}\label{grammar-Stirling}
D_{G}^n(x)=C_n(x,y,z),
\end{equation}
where $G=\{x \rightarrow xyz, y\rightarrow xyz, z\rightarrow xyz\}$.
B\'ona~\cite{Bona08} independently noted that the statistics $\des,~\asc$ and $\plat$ are equidistributed on $\mqn$.
In other words,
$$C_n(x)=\sum_{\sigma\in\man}x^{\des(\sigma)}=\sum_{\sigma\in\man}x^{\asc(\sigma)}=\sum_{\sigma\in\man}x^{\plat(\sigma)}.$$
By constructing an urn model, Janson~\cite[Theorem~2.1]{Janson08} rediscovered
the symmetry of $C_n(x,y,z)$.
In~\cite{Haglund12}, Haglund-Visontai introduced a refinement of the polynomial $C_n(x,y,z)$ by indexing each ascent,
descent and plateau by the values where they appear.
Using~\eqref{grammar-Stirling}, Chen-Fu~\cite{Chen22}
found that $C_n(x,y,z)$ is $e$-positive, i.e.,
\begin{equation}\label{Cnxyz}
C_n(x,y,z)=\sum_{i+2j+3k=2n+1}\gamma_{n,i,j,k}(x+y+z)^{i}(xy+yz+zx)^{j}(xyz)^k,
\end{equation}
where the coefficient $\gamma_{n,i,j,k}$ counts 0-1-2-3 increasing plane trees
on $[n]$ with $k$ leaves, $j$ degree one vertices and $i$ degree two vertices. The reader is referred to~\cite{Ma23} for the recent progress on this subject.
In the sequel, we shall present four interpretations of the second-order Eulerian polynomials.

The {\it bivariate second-order Eulerian polynomials} can be defined by
$$C_n(x,y)=\sum_{\sigma\in\mqn}x^{\des(\sigma)}y^{2n+1-\des(\sigma)}=\sum_{\sigma\in\mqn}x^{\des(\sigma)}y^{\asc(\sigma)+\plat(\sigma)}.$$
Below is a deep connection between the Eulerian polynomials and the second-order Eulerian polynomials.
\begin{theorem}\label{thm-Cnxy}
Let $A_n(x,y)$ be the bivariate Eulerian polynomial, which is defined by~\eqref{Anxy}. Then
$$C_n(x,y)=\sum_{T\in\operatorname{SYT}{(n)}}\prod_{i=1}^n\sigma_i(T)\left(\prod_{i=1}^nA_i(x,y)\right)^{w_i(T)}y^{n+1-\ell(\lambda(T))}.$$
\end{theorem}
\begin{proof}
Recall Example~\ref{ExG4}. Let $G'=\{x\rightarrow xy,~y\rightarrow xy\}$. We notice that
\begin{equation}\label{xyCn02}
(yD_{G'})^n(y)=y^{2n+1}C_n\left(\frac{x}{y}\right).
\end{equation}
Taking $c=y$, it follows from Proposition~\ref{grammar03} that $c_i=D_{G'}^i(y)=A_i(x,y)$ for $i\geqslant 1$.
By~\eqref{cdnc-Thm}, we find that
\begin{equation*}
C_n(x)=\sum_{T\in\operatorname{SYT}{(n)}}\left(\prod_{i=1}^n\sigma_i(T)c_i^{w_i(T)}\right)c^{n+1-\ell(\lambda(T))}{\bigg|}_{c=y,~c_{i}=A_i(x,y)},
\end{equation*}
which gives the desired formula. This completes the proof.
\end{proof}

The first few $A_n(x,y)$ are given as follows:
\begin{equation*}
A_1(x,y)=xy,~
A_2(x,y)=xy^2+x^2y,~
A_3(x,y)=xy^3+4x^2y^2+x^3y.
\end{equation*}
An illustration of Theorem~\ref{thm-Cnxy} is given by Table~\ref{tab:dummy-c3xy}.
\renewcommand{\arraystretch}{2}
\begin{center}
\begin{table}[ht!]
  \caption{The computation of $C_3(x,y)=xy^6+8x^2y^5+6x^3y^4$.}\label{tab:dummy-c3xy}
  {
\begin{tabular}{ccccc|c}
$T$&&${\sigma_i(T)},~w_i(T)$&&&enumerator\\
\hline
    \begin{ytableau}
    *(gray!20) 3\\
    *(gray!20) 2\\
    *(gray!20) 1\\
    *(gray!20) 0 & 1 & 2 &3  \\
    \end{ytableau}
    & $\Longrightarrow$
    &
   $\substack{\sigma_1(T)=\sigma_2(T)=\sigma_3(T)=1\\ w_1(T)=w_2(T)=0,~w_3(T)=1}$
    & &
    &$(xy^3+4x^2y^2+x^3y)y^3$\\
\hline
    \begin{ytableau}
    *(gray!20) 3\\
    *(gray!20) 2\\
    *(gray!20) 1 & 2\\
    *(gray!20) 0 & 1 & 3  \\
    \end{ytableau}
    & $\Longrightarrow$
    &
  $\substack{\sigma_1(T)=\sigma_2(T)=1,~\sigma_3(T)=2\\w_1(T)=w_2(T)=1,~w_3(T)=0}$
    &
    &
    &$2(xy)(xy^2+x^2y)y^2$\\
\hline
    \begin{ytableau}
    *(gray!20) 3\\
    *(gray!20) 2\\
    *(gray!20) 1 & 3\\
    *(gray!20) 0 & 1 & 2  \\
    \end{ytableau}
    & $\Longrightarrow$
    &
   $\substack{\sigma_1(T)=\sigma_2(T)=1,~\sigma_3(T)=2\\w_1(T)=w_2(T)=1,~w_3(T)=0}$
    &
    &
    &$2(xy)(xy^2+x^2y)y^2$\\
\hline
    \begin{ytableau}
    *(gray!20) 3\\
    *(gray!20) 2 & 3\\
    *(gray!20) 1 & 2\\
    *(gray!20) 0 & 1\\
    \end{ytableau}
    & $\Longrightarrow$
    &
  $\substack{\sigma_1(T)=\sigma_2(T)=\sigma_3(T)=1\\w_1(T)=3,~w_2(T)=w_3(T)=0}$
    &
    &
    &$(xy)^3y$\\
    \hline
\end{tabular}}
\end{table}
\end{center}

Recall Example~\ref{ExG4}.
The grammatical description of the second-order Eulerian polynomials $C_n(x)$ can also be restated as follows:
\begin{equation}\label{xyCn}
(xD_{G})^n(x)=y^{2n+1}C_n\left(\frac{x}{y}\right),~\text{where $G=\{x\rightarrow y^2,~y\rightarrow y^2\}$.}
\end{equation}

\begin{theorem}\label{cnx01}
Let $C_n(x)$ be the second-order Eulerian polynomials. We have
$$C_n(x)=\sum_{T\in\operatorname{SYT}{(n)}}\left(\prod_{i=1}^n\sigma_i(T){i!}^{w_i(T)}\right)x^{n+1-\ell(\lambda(T))}.$$
\end{theorem}
\begin{proof}
From~\eqref{xyCn}, we see that
$C_n(x)=(xD_{G})^n(x){\bigg|}_{y=1}$, where $G=\{x\rightarrow y^2,~y\rightarrow y^2\}$.
Taking $c=x$, then $c_i=D_{G}^i(c)=D_{G}^i(x)=i!y^{i+1}$ for $i\geqslant 1$.
By~\eqref{cdnc-Thm}, we get
\begin{equation*}
C_n(x)=\sum_{T\in\operatorname{SYT}{(n)}}\left(\prod_{i=1}^n\sigma_i(T)c_i^{w_i(T)}\right)c^{n+1-\ell(\lambda(T))}{\bigg|}_{c=x,~c_{i}=i!y^{i+1},~y=1},
\end{equation*}
which yields the desired result. This completes the proof.
\end{proof}

\begin{theorem}\label{cnx02}
For the trivariate second-order Eulerian polynomials, we have
$$C_{n+1}(x,y,z)=\sum_{T\in\operatorname{SYT}{(n;3)}}\prod_{i=1}^n\sigma_i(T){c_1}^{w_1(T)}{c_2}^{w_2(T)}6^{w_3(T)}
(xyz)^{n+1-\ell(\lambda(T))},$$
where $c_1=xy+yz+xz$ and $c_2=2x+2y+2z$.
\end{theorem}
\begin{proof}
It follows from~\eqref{grammar-Stirling} that
$(xyzD_{G})^n(xyz)=C_{n+1}(x,y,z)$, where $$G=\{x\rightarrow 1,~y\rightarrow 1, z\rightarrow 1\}.$$
Setting $c=xyz$, we get that $c_1=xy+yz+xz$, $c_2=2x+2y+2z$, $c_3=D_{G}(2x+2y+2z)=6$,
and $c_i=0$ for $i\geqslant 4$.
Substituting $c=xyz,~c_{1}=xy+yz+xz,~c_{2}=2x+2y+2z$, $c_3=6$, and $c_i=0$ for $i\geqslant 4$ into the following expression:
\begin{equation*}
C_{n+1}(x,y,z)=\sum_{T\in\operatorname{SYT}{(n)}}\left(\prod_{i=1}^n\sigma_i(T)c_i^{w_i(T)}\right)
c^{n+1-\ell(\lambda(T))},
\end{equation*}
we obtain the desired formula. This completes the proof.
\end{proof}

In the following, we shall deduce another interpretation of $C_n(x)$ in terms of standard Young tableaux.
From Example~\ref{ExG4}, we see that let $G'=\{x\rightarrow x,~y\rightarrow x\}$, then we have
\begin{equation}\label{y2DG12}
(y^2D_{G'})^n(y)=y^{2n+1}C_n\left(\frac{x}{y}\right).
\end{equation}

Let's rewrite $(c^2D)^n c$ as follows:
$$\left(c_{(2n)}c_{(2n-1)}D_{(n)}\right)\left(c_{(2n-2)}c_{(2n-3)}D_{(n-1)}\right) \cdots \left(c_{(4)}c_{(3)}D_{(2)}\right) \left(c_{(2)}c_{(1)}D_{(1)}\right) c_{(0)}.$$
As an variant of Definition~\ref{defOWP}, we introduce the following definition.
\begin{definition}\label{defOWP02}
Let $\operatorname{{\overline{OWP}}}_{n}$
denote the collection of ordered weak set partitions of $[n]$ into $2n+1$ blocks, i.e., $[n]=B_0\cup B_1\cup\overline{B}_1\cup B_2\cup\overline{B}_2\cdots\cup B_n\cup\overline{B}_n$, and for which the following conditions hold:
$(a)$ $1\in B_0$; $(b)$ if $B_i$ or $\overline{B}_i$ is nonempty, then its minimum element larger than $i$.
\end{definition}

\begin{figure}[ht!]
\renewcommand{\arraystretch}{2}
\begin{center}
\begin{tabular}{c}
    \begin{ytableau}
    *(gray!20) \bar{1}&\\
    *(gray!20) 1&\\
    *(gray!20) 0 & 1 \\
    \end{ytableau}\\
\end{tabular}~~\quad
\begin{tabular}{c}
    \begin{ytableau}
    *(gray!20) \bar{2}&\\
    *(gray!20) 2&\\
    *(gray!20) \bar{1}&\\
    *(gray!20) 1&\\
    *(gray!20) 0 & 1 & 2  \\
    \end{ytableau}~~\quad
    \begin{ytableau}
    *(gray!20) \bar{2}&\\
    *(gray!20) 2&\\
    *(gray!20) \bar{1}&\\
    *(gray!20) 1 & 2\\
    *(gray!20) 0 & 1\\
    \end{ytableau}\\
\end{tabular}~~\quad
\begin{tabular}{c}
    \begin{ytableau}
    *(gray!20) \bar{2}&\\
    *(gray!20) 2&\\
    *(gray!20) \bar{1}&2\\
    *(gray!20) 1&\\
    *(gray!20) 0 & 1 \\
    \end{ytableau}\\
\end{tabular}
\end{center}
\caption{Represents of $(c^2D)~c = c^2c_1$ and $(c^2D)^2~c=c^4c_2+2c^3c_1^2$ by using $\operatorname{YWCT}$}
\end{figure}

\begin{figure}[!ht]\label{figc2D3c}
\renewcommand{\arraystretch}{2}
\begin{center}
\begin{tabular}{|c|c|c|c|c|}
\hline
$T \in \SYT(3)$&
    \begin{ytableau}
    *(gray!20) \bar{3}\\
    *(gray!20) 3\\
    *(gray!20) \bar{2}\\
    *(gray!20) 2\\
    *(gray!20) \bar{1}\\
    *(gray!20) 1\\
    *(gray!20) 0 & 1 & 2 &3  \\
    \end{ytableau}
    &
    \begin{ytableau}
    *(gray!20) \bar{3}\\
    *(gray!20) 3\\
    *(gray!20) \bar{2}\\
    *(gray!20) 2\\
    *(gray!20) \bar{1}\\
    *(gray!20) 1 & 3\\
    *(gray!20) 0 & 1 & 2  \\
    \end{ytableau}
    &
    \begin{ytableau}
    *(gray!20) \bar{3}\\
    *(gray!20) 3\\
    *(gray!20) \bar{2}\\
    *(gray!20) 2\\
    *(gray!20) \bar{1}\\
    *(gray!20) 1 & 2\\
    *(gray!20) 0 & 1 & 3  \\
    \end{ytableau}
    &
    \begin{ytableau}
    *(gray!20) \bar{3}\\
    *(gray!20) 3\\
    *(gray!20) \bar{2}\\
    *(gray!20) 2\\
    *(gray!20) \bar{1} & 3\\
    *(gray!20) 1 & 2\\
    *(gray!20) 0 & 1\\
    \end{ytableau}\\
\hline
$\prod_{i=1}^3\delta_i(T) \cdot w(T)$&
$c^6c_3$
&
$4c^5c_2c_1$
&
$2\cdot 2c^5c_2c_1$
&
$2\cdot 3c^4c_1^3$\\
\hline
$C_3(x)$&
$x$
&
$4x^2$
&
$4x^2$
&
$6x^3$\\
\hline
\end{tabular}
\end{center}
\caption{$(c^2D)^3~c =c^6c_3+8c^5c_2c_1+6c^4c_1^3$ and $C_3(x)=x+8x^2+6x^3$}\label{figc2D3c}
\end{figure}

For any $p\in\operatorname{{\overline{OWP}}}_{n}$, the weight function of the corresponding standard Young tableau is given as follows:
\begin{equation}\label{wT02}
w(T)=c^{2n+1-\ell(\lambda(T))}\prod_{i=1}^nc_i^{w_i(T)}.
\end{equation}
In Figure~\ref{figc2D3c}, we present illustrations of $w(T)$ for all $T\in\SYT(3)$ as well as the so called second-order box sorting indices.
In the same way as in the proof of Theorem~\ref{TSYTncol}, it is routine to check the following result, and we omit the proof for simplify.
\begin{theorem}\label{thmc2D}
For $n\geqslant 1$, we have $(c^2D)^nc=\sum_{p\in\operatorname{\overline{OWP}}_{n}}w(p)$. Moreover,
\begin{equation}\label{c2D}
(c^2D)^n c=\sum_{T\in\operatorname{SYT}{(n)}}\left(\prod_{i=1}^n\delta_i(T)c_i^{w_i(T)}\right)c^{2n+1-\ell(\lambda(T))}.
\end{equation}
We call $\delta_i(T)$ the second-order box sorting index of the entry $i$, which is defined by
\begin{equation*}
\delta_i(T)=\left\{
  \begin{array}{ll}
   2i-\operatorname{col}_1(T_i), & \hbox{if $i$ is in the first column of $T$;} \\
    \operatorname{col}_{k}(T_i)-\operatorname{col}_{k+1}(T_i)+1, & \hbox{if $i$ is in the $(k+1)$-th column of $T$, where $k\geqslant 1$.}
  \end{array}
\right.
\end{equation*}
\end{theorem}

\begin{theorem}\label{cnx03}
Let $C_n(x)$ be the second-order Eulerian polynomials.
Then we have
$$C_n(x)=\sum_{T \in \SYT(n)} \prod_{i=1}^n \delta_i(T)  x^{\ell(\lambda(T))}.$$
\end{theorem}
\begin{proof}
From~\eqref{y2DG12}, it is natural to set $c=y$. Then $c_i=D_{G'}^i(c)=D_{G'}^i(y)=x$ for any $i\geqslant 1$,
where $G'=\{x\rightarrow x,~y\rightarrow x\}$.
Using~\eqref{c2D}, we find that
\begin{align*}
C_n(x)&=\sum_{T\in\operatorname{SYT}{(n)}}\left(\prod_{i=1}^n\delta_i(T)c_i^{w_i(T)}\right)c^{2n+1-\ell(\lambda(T))}{\bigg|}_{c=y,~c_{i}=x,~y=1},
\end{align*}
which yields the desired result. This completes the proof.
\end{proof}

Let $k$ be a given positive integer.
Let $\operatorname{{\overline{OWP}}}_{n}^{(k)}$
denote the collection of ordered weak set partitions of $[n]$ into $kn+1$ blocks, i.e., $[n]=B_0\cup B_{11}\cup{B}_{12}\cup\cdots \cup B_{1k}\cup B_{21}\cup{B}_{22}\cup\cdots\cup B_{2k}\cup \cdots \cup B_{n1}\cup B_{n2}\cup\cdots \cup{B}_{nk}$, and for which the following conditions hold:
$(a)$ $1\in B_0$; $(b)$ if $B_{ik}$ is nonempty, then its minimum element larger than $i$, where $1\leqslant i\leqslant n$.
For any $p\in\operatorname{{\overline{OWP}}}_{n}^{(k)}$, the weight function of the corresponding standard Young tableau is defined by
\begin{equation*}\label{wT02}
w(T)=c^{kn+1-\ell(\lambda(T))}\prod_{i=1}^nc_i^{w_i(T)}.
\end{equation*}
Along the same lines as discussed in Section~\ref{sec05},
it is routine to verify that
\begin{equation}\label{ckD}
(c^kD)^n c=\sum_{T\in\operatorname{SYT}{(n)}}\left(\prod_{i=1}^n\delta_i^{(k)}(T)c_i^{w_i(T)}\right)c^{kn+1-\ell(\lambda(T))},
\end{equation}
where
\begin{equation*}
\delta_i^{(k)}(T)=\left\{
  \begin{array}{ll}
   ki-\operatorname{col}_1(T_i)-(k-2), & \hbox{if $i$ is in the first column of $T$;} \\
    \operatorname{col}_{k}(T_i)-\operatorname{col}_{k+1}(T_i)+1, & \hbox{if $i$ is in the $(k+1)$-th column of $T$, where $k\geqslant 1$.}
  \end{array}
\right.
\end{equation*}
It is clear that $\delta_i^{(1)}(T)=\sigma_i(T)$ and $\delta_i^{(2)}(T)=\delta_i(T)$.

A {\it $k$-Stirling permutation} of order $n$ is a multiset permutation of $\{1^k,2^k,\ldots,n^k\}$
with the property that all elements between two occurrences of $i$ are at least $i$ for all $i\in [n]$, see~\cite{Lin21,Ma23} for the recent study on $k$-Stirling permutations and their variants.
Let $\mqn(k)$ be the set of $k$-Stirling permutations of order $n$. It is clear that $\mqn(1)=\msn,~\mqn(2)=\mqn$.
We say that an index $i\in [kn]$ is a {\it descent} of $\sigma$ if
$\sigma_{i}>\sigma_{i+1}$ or $i=kn$. Let $\mqn(k)$ be the set of $k$-Stirling permutations of order $n$. The {\it $k$-order Eulerian polynomials} are defined by $$C_n(x;k)=\sum_{\sigma\in\mqn(k)}x^{\des(\pi)},~C_0(x;k)=1.$$
Following~\cite[Lemma~1]{Dzhumadil14},
the polynomials $C_n(x;k)$ satisfy the recurrence relation
\begin{equation}\label{recuCnx}
C_{n+1}(x;k)=(1+kn)xC_n(x;k)+x(1-x)\frac{\mathrm{d}}{\mathrm{d}x}C_n(x;k),~C_0(x;k)=1.
\end{equation}
Note that $C_n(x;1)=A_n(x)$ and $C_n(x;2)=C_n(x)$.
By~\eqref{recuCnx}, it is easy to verify the following.
\begin{lemma}\label{lemmaCnk}
We have
$$\left(\frac{x}{(1-x)^k}\frac{\mathrm{d}}{\mathrm{d}x}\right)^n\frac{x}{1-x}=\frac{C_n(x;k+1)}{(1-x)^{n+kn+1}}.$$
\end{lemma}

By the differential operator method introduced in Section~\ref{sec03},
setting $$T=\frac{x}{(1-x)^k}\frac{\mathrm{d}}{\mathrm{d}x},~a=\frac{x}{1-x},~b=\frac{1}{1-x},$$
we obtain $T(a)=T(b)=ab^{k+1}$.
By Lemma~\ref{lemmaCnk}, we find that for $n\geqslant 1$,
\begin{equation}\label{DGa-Cnk}
D_G^n(a)=D_G^n(b)=b^{kn+1}C_n\left(\frac{a}{b};k\right),~{\text{where $G=\{a\rightarrow ab^k,~b\rightarrow ab^k\}$}}.
\end{equation}
\begin{theorem}
Let $k$ be a given positive integer.
We have
$$C_n(x;k)=\sum_{T \in \SYT(n)} \prod_{i=1}^n \delta_i^{(k)}(T)  x^{\ell(\lambda(T))},$$
where $\ell(\lambda(T))$ is the number of rows of $T$ and $\delta_i^{(k)}(T)$ is defined by
\begin{equation*}
\delta_i^{(k)}(T)=\left\{
  \begin{array}{ll}
   ki-\operatorname{col}_1(T_i)-(k-2), & \hbox{if $i$ is in the first column of $T$;} \\
    \operatorname{col}_{k}(T_i)-\operatorname{col}_{k+1}(T_i)+1, & \hbox{if $i$ is in the $(k+1)$-th column of $T$, where $k\geqslant 1$.}
  \end{array}
\right.
\end{equation*}
\end{theorem}
\begin{proof}
Let $G'=\{x\rightarrow x,~y\rightarrow x\}$. We note that~\eqref{DGa-Cnk} can be rewritten as follows:
$$(y^kD_{G'})y=y^{kn+1}C_n\left(\frac{x}{y};k\right).$$
Setting $c=y$, we get $c_i=D_{G'}^i(c)=x$ for all $i\geqslant 1$.
By~\eqref{c2D}, we find that
\begin{align*}
C_n(x;k)&=\sum_{T\in\operatorname{SYT}{(n)}}\left(\prod_{i=1}^n\delta_i^{(k)}(T)c_i^{w_i(T)}\right)c^{kn+1-\ell(\lambda(T))}{\bigg|}_{c=y,~c_i=x,~y=1}\\
&=\sum_{T\in\operatorname{SYT}{(n)}}\prod_{i=1}^n\delta_i^{(k)}(T)x^{w_i(T)},
\end{align*}
which yields the desired formula. This completes the proof.
\end{proof}
\subsection{Narayana polynomials of types $A$ and $B$}
\hspace*{\parindent}

It is well known that Catalan numbers and central binomial coefficients have the following famous expressions (see~\cite{Chen11,Coker03}):  $$C_n:=\frac{1}{n+1}\binom{2n}{n}=\sum_{k=0}^{n-1}\frac{1}{n}\binom{n}{k+1}\binom{n}{k},~\binom{2n}{n}=\sum_{k=0}^{n}{\binom{n}{k}}^2.$$
In~\cite{Fomin03}, Fomin and Zelevinsky defined the (generalized) Narayana numbers $N_k(\Phi)$
for an arbitrary root system $\Phi$ as the entries of the $h$-vector of the simplicial complex
dual to the corresponding generalized associahedron.
Let $N(\Phi,x)=\sum_{k=0}^nN_k(\Phi)x^k$. For the classical Weyl groups, the Narayana polynomials of types $A$ and $B$ are given as follows:
\begin{equation*}
N(A_n,x)=\sum_{k=0}^n\frac{1}{n+1}\binom{n+1}{k+1}\binom{n+1}{k}x^k,~
N(B_n,x)=\sum_{k=0}^n{\binom{n}{k}}^2x^k.
\end{equation*}

The following two results will be used in our discussion.
\begin{lemma}[{\cite[Theorem~9]{Ma1902}}]\label{MMY}
Let $G=\{x\rightarrow x^2y^3,~y\rightarrow x^3y^2\}$.
Then we have
\begin{equation}\label{DGx2}
D_G^n(x^2)=(n+1)!x^{n+2}y^{3n}N\left(A_{n-1},\frac{x^2}{y^2}\right),
\end{equation}
\begin{equation}\label{DGxy}
D_G^n(xy)=n!x^{n+1}y^{3n+1}N\left(B_n,\frac{x^2}{y^2}\right).
\end{equation}
\end{lemma}

\begin{lemma}[{\cite[Theorem~10]{Ma131}}]\label{MMY02}
Let $G=\{x\rightarrow xy^2,~y\rightarrow x^2y\}$. For any $n\geqslant 1$, we have
\begin{equation}\label{DGx202}
D_G^n(x^2)=2^ny^{2n+2}A_n\left(\frac{x^2}{y^2}\right),
\end{equation}
\begin{equation}\label{DGxy02}
D_G^n(xy)=xy^{2n+1}B_n\left(\frac{x^2}{y^2}\right),
\end{equation}
where $A_n(x)$ and $B_n(x)$ are the types $A$ and $B$ Eulerian polynomials.
\end{lemma}

We can now give the following result.
\begin{theorem}
The Narayana polynomials of types $A$ and $B$ have the following expansions:
\begin{align*}
(n+1)!N\left(A_{n-1},x\right)&=\sum_{T\in\operatorname{SYT}{(n)}}\left(\prod_{i=1}^n\sigma_i(T)\right)\left(\prod_{i=1}^n{~i!\sum_{j\geqslant 0}\binom{i+1}{2j-i}x^{j-\frac{i}{2}}}\right)^{w_i(T)}x^{\frac{n}{2}-\ell(\lambda(T))},\\
n!N\left(B_n,x\right)&=\sum_{T\in\operatorname{SYT}{(n)}}\left(\prod_{i=1}^n\delta_i(T)c_i^{w_i(T)}\right)
x^{\frac{n}{2}-\frac{1}{2}\ell(\lambda(T))},
\end{align*}
where $c_{2i-1}=4^{i-1}(x+1),~c_{2i}=4^i\sqrt{x}$, $\sigma_i(T)$ is the box sorting index and $\delta_i(T)$ is the second-order box sorting index.
\end{theorem}
\begin{proof}
(A) Let $G'=\{x\rightarrow y^3,~y\rightarrow xy^2\}$. Then~\eqref{DGx2} can be restated as follows:
$$\left(x^2D_{G'}\right)^nx^2=(n+1)!x^{n+2}y^{3n}N\left(A_{n-1},\frac{x^2}{y^2}\right).$$
Set $c=x^2$. By induction, it is routine to verify that
$$c_i=D_{G'}^i(c)=D_{G'}^i(x^2)=i!\sum_{j\geqslant 0}\binom{i+1}{2j-i}x^{2j-i}y^{3i-2j+2}~{\text{for $i\geqslant 1$}}.$$
where $\binom{i+1}{2j-i}$ counts binary words of length $i$ with $j$ strictly increasing runs (see~\cite[A119900]{Sloane}).
It follows from~\eqref{cdnc-Thm} that
\begin{align*}
&(n+1)!x^{n+2}N\left(A_{n-1},x^2\right)\\&=\sum_{T\in\operatorname{SYT}{(n)}}\left(\prod_{i=1}^n\sigma_i(T)c_i^{w_i(T)}\right)c^{n+1-\ell(\lambda(T))}{\bigg|}_{c=x^2,~c_i=i!\sum_{j\geqslant 0}\binom{i+1}{2j-i}x^{2j-i}}\\
&=\sum_{T\in\operatorname{SYT}{(n)}}\left(\prod_{i=1}^n\sigma_i(T)\right)\left(\prod_{i=1}^n{~i!\sum_{j\geqslant 0}\binom{i+1}{2j-i}x^{2j-i}}\right)^{w_i(T)}x^{2n+2-2\ell(\lambda(T))}.
\end{align*}
After simplifying, we get the desired formula.

(B) Let $G''=\{x\rightarrow y,~y\rightarrow x\}$. Then~\eqref{DGxy} can be restated as follows:
$$\left(x^2y^2D_{G''}\right)^nxy=n!x^{n+1}y^{3n+1}N\left(B_n,\frac{x^2}{y^2}\right).$$
Setting $c=xy$, it is clear that for any $i\geqslant 1$, we have
$$c_{2i-1}=D_{G''}^{2i-1}(c)=D_{G''}^{2i-1}(xy)=4^{i-1}(x^2+y^2),$$
$$c_{2i}=D_{G''}^{2i}(c)=D_{G''}\left(4^{i-1}(x^2+y^2)\right)=4^{i}xy.$$
By~\eqref{c2D}, we obtain that
\begin{equation*}
n!x^{n+1}N\left(B_n,x^2\right)=\sum_{T\in\operatorname{SYT}{(n)}}\left(\prod_{i=1}^n\delta_i(T)c_i^{w_i(T)}\right)
x^{2n+1-\ell(\lambda(T))}{\bigg|}_{{\substack{c_{2i-1}=4^{i-1}(x^2+1)\\c_{2i}=4^ix}}},
\end{equation*}
which yields the desired result.
This completes the proof.
\end{proof}
\renewcommand{\arraystretch}{2}
\begin{center}
\begin{table}[ht!]
  \caption{An illustration of $3!N(B_3,x)=6(1+9x+9x^2+x^3)$.}\label{tab:dummy-07BB}
  {
\begin{tabular}{ccccc|c}
$T$&&${\sigma_i(T)},~w_i(T)$&&&enumerator\\
\hline
    \begin{ytableau}
    *(gray!20) 3\\
    *(gray!20) 2\\
    *(gray!20) 1\\
    *(gray!20) 0 & 1 & 2 &3  \\
    \end{ytableau}
    & $\Longrightarrow$
    &
   $\substack{\sigma_1(T)=\sigma_2(T)=\sigma_3(T)=1\\ w_1(T)=w_2(T)=0,~w_3(T)=1}$
    & &
    &$1+23x+23x^2+x^3$\\
\hline
    \begin{ytableau}
    *(gray!20) 3\\
    *(gray!20) 2\\
    *(gray!20) 1 & 2\\
    *(gray!20) 0 & 1 & 3  \\
    \end{ytableau}
    & $\Longrightarrow$
    &
  $\substack{\sigma_1(T)=\sigma_2(T)=1,~\sigma_3(T)=2\\w_1(T)=w_2(T)=1,~w_3(T)=0}$
    &
    &
    &$2(1+x)(1+6x+x^2)$\\
\hline
    \begin{ytableau}
    *(gray!20) 3\\
    *(gray!20) 2\\
    *(gray!20) 1 & 3\\
    *(gray!20) 0 & 1 & 2  \\
    \end{ytableau}
    & $\Longrightarrow$
    &
   $\substack{\sigma_1(T)=\sigma_2(T)=1,~\sigma_3(T)=2\\w_1(T)=w_2(T)=1,~w_3(T)=0}$
    &
    &
    &$2(1+x)(1+6x+x^2)$\\
\hline
    \begin{ytableau}
    *(gray!20) 3\\
    *(gray!20) 2 & 3\\
    *(gray!20) 1 & 2\\
    *(gray!20) 0 & 1\\
    \end{ytableau}
    & $\Longrightarrow$
    &
  $\substack{\sigma_1(T)=\sigma_2(T)=\sigma_3(T)=1\\w_1(T)=3,~w_2(T)=w_3(T)=0}$
    &
    &
    &$(1+x)^3$\\
    \hline
\end{tabular}}
\end{table}
\end{center}
We end this paper by giving the following result, which gives a deep connection between the type $B$ Narayana polynomials and the type $B$ Eulerian polynomials.
\begin{theorem}\label{NBthm}
Let $N\left(B_n,{x}\right)$ be the type $B$ Narayana polynomials, and let $B_n(x)$ be the type $B$ Eulerian polynomials.
Then we have
$$n!N\left(B_n,{x}\right)=\sum_{T\in\operatorname{SYT}{(n)}}\prod_{i=1}^n\sigma_i(T)\left(B_i(x)\right)^{w_i(T)}.$$
\end{theorem}
\begin{proof}
Let $G'''=\{x\rightarrow xy^2,~y\rightarrow x^2y\}$. Then~\eqref{DGxy} can be restated as follows:
$$\left(xyD_{G'''}\right)^nxy=n!x^{n+1}y^{3n+1}N\left(B_n,\frac{x^2}{y^2}\right).$$
Setting $c=xy$, by~\eqref{DGxy02}, we get
\begin{equation*}
c_i=D_{G'''}^i(c)=D_{G'''}^i(xy)=xy^{2i+1}B_i\left(\frac{x^2}{y^2}\right).
\end{equation*}
It follows from~\eqref{cdnc-Thm} that
\begin{align*}
n!x^{n+1}N\left(B_n,{x^2}\right)&=\sum_{T\in\operatorname{SYT}{(n)}}\left(\prod_{i=1}^n\sigma_i(T)c_i^{w_i(T)}\right)c^{n+1-\ell(\lambda(T))}{\bigg|}_{c=x,~c_i=xB_i(x^2)}\\
&=\sum_{T\in\operatorname{SYT}{(n)}}\left(\prod_{i=1}^n\sigma_i(T)c_i^{w_i(T)}\right)x^{n+1-\ell(\lambda(T))}{\bigg|}_{c_i=xB_i(x^2)},
\end{align*}
and so we arrive at
$$n!N\left(B_n,{x}\right)=\sum_{T\in\operatorname{SYT}{(n)}}\prod_{i=1}^n\sigma_i(T)\left(B_i(x)\right)^{w_i(T)}.$$
This completes the proof.
\end{proof}

The first few type $B$ Eulerian polynomials are listed below:
$$B_1(x)=1+x,~B_2(x)=1+6x+x^2,~B_3(x)=1+23x+23x^2+x^3.$$
In Table~\ref{tab:dummy-07BB}, we give an illustration of Theorem~\ref{NBthm}.
\bibliographystyle{amsplain}

\end{document}